\documentclass[11pt]{article}
\usepackage{amssymb}
\usepackage{mathrsfs}
\usepackage{latexsym,amsmath,amssymb,amsfonts,epsfig,graphicx,cite,psfrag}
\usepackage{eepic,color,colordvi,amscd}
\usepackage{ebezier}
\usepackage{verbatim}
\usepackage{comment}
\usepackage{subfigure}

\usepackage{amsthm}
\usepackage{setspace}

\newcommand{\pf}{\noindent {\it Proof.} }

\newtheorem{theorem}{Theorem}[section]
\newtheorem{lemma}[theorem]{Lemma}
\newtheorem{coro}[theorem]{Corollary}

\theoremstyle{remark}

\def\qed{\hfill \rule{4pt}{7pt}}

\renewcommand{\d}{\mathrm{d}}

\newcommand{\N}{{\mathbb{N}}}

\newcommand{\E}{{\mathbb{E}}}
\newcommand{\Z}{{\mathbb{Z}}}
\renewcommand{\P}{{\mathbb{P}}}

\renewcommand{\S}{{\mathcal{S}}}

\newcommand{\F}{{\mathcal{F}}}
\newcommand{\T}{{\mathcal{T}}}
\newcommand{\I}{{\mathcal{I}}}
\newcommand{\M}{{\mathcal{M}}}

\newcommand{\bd}{{\mathrm{DEG}}}

\begin{document}

\title{Nearly perfect matchings in uniform hypergraphs}

\author{Hongliang Lu\footnote{Partially supported by the National Natural
Science Foundation of China under grant No.11471257}\\
School of Mathematics and Statistics\\
Xi'an Jiaotong University\\
Xi'an, Shaanxi 710049, China\\
\medskip \\
 Xingxing Yu\footnote{Partially supported by NSF grant DMS-1600738} ~and Xiaofan Yuan\footnote{Partially supported by NSF grant DMS-1600738
 (through X. Yu) and Research Assistantship from ACO at Georgia Tech}\\
School of Mathematics\\
Georgia Institute of Technology\\
Atlanta, GA 30332}

\date{}

\maketitle

\date{}

\maketitle

\begin{abstract}
We prove that,
for any integers $k,l$ with $k\ge 3$ and $k/2<l\le  k-1$, there exists a positive real $\mu$ such that, for all integers $m,n$ satisfying
\[
	\frac{n}{k}-\mu n\le m\le \frac{n}{k} -1- \left(1-\frac{l}{k}\right)\left\lceil\frac{k-l}{2l-k}\right\rceil
\]
and $n$ sufficiently large,
if $H$ is  a $k$-uniform hypergraph on $n$ vertices and  $\delta_{l}(H)>{n-l\choose k-l}-{(n-l)-m\choose k-l}$,
then $H$ has a matching of size $m+1$.
This improves upon an earlier result of  H\`{a}n, Person, and Schacht
for the range $k/2<l\le k-1$.  In many cases, our result gives tight bound
on $\delta_l(H)$ for near perfect matchings (e.g., when $l\ge 2k/3$, $n\equiv r \pmod k$, $0\le r<k$, and $r+l\ge k$ we can take $m=\lceil n/k\rceil -2$).
When $k=3$, using an absorbing lemma of H\`{a}n, Person, and Schacht,
our proof also implies a result of K\"uhn, Osthus, and Treglown (and,
independently, of Khan) on perfect matchings in 3-uniform hypergraphs.
\end{abstract}

\newpage

\section{Introduction}

A \emph{hypergraph} $H$ consists of a vertex set $V(H)$ and an edge
set $E(H)$ whose members are subsets of $V(H)$. For a  positive integer
$k$, a hypergraph $H$ is
\emph{$k$-uniform} if $E(H)\subseteq {V(H)\choose k}$, and a
$k$-uniform hypergraph is also called a \emph{$k$-graph}.

Let $H$ be a hypergraph. A {\it matching} in $H$ is a set of pairwise
disjoint edges of $H$. (If $M$ is a matching in $H$, we use
$V(M)$ to denote the vertices of $H$ covered by $M$.)
The size of a largest matching in $H$ is denoted by $\nu(H)$, known as the {\it matching number} of $H$. A
matching in $H$  is {\it perfect} if it covers all vertices of $H$.
A matching is {\it nearly perfect} in
$H$ if it covers all but a constant number of vertices.
Moreover, a matching in a $k$-graph is {\it near perfect}
if it covers all but at most $k$ vertices.

We are interested in degree conditions for the
existence of a nearly perfect matching  in a hypergraph.
Let $H$ be a hypergraph.  For any $T\subseteq V(H)$, we use $d_H(T)$ to denote the {\it degree}
of $T$ in $H$, i.e., the number of edges of $H$ containing $T$.
Let $l$ be a non-negative integer.
Then $\delta_l(H):= \min\left\{d_H(T): T\in {V(H)\choose l}\right\}$ is the minimum
{\it $l$-degree} of $H$. $\delta_1(H)$ is often called the minimum {\it
vertex degree} of $H$ and $\delta_{k-1}(H)$ is known as the minimum
{\it codegree} of $H$. Note that $\delta_0(H)$ is the number of edges in $H$.

Bollob\'{a}s, Daykin, and Erd\H{o}s \cite{BDE76} considered minimum
vertex degree conditions for matchings in $k$-graphs.
They proved that if $H$ is a $k$-graph of order $n\geq 2k^2(m+2)$ and
 $\delta_1(H)>{{n-1}\choose {k-1}}-{{n-m}\choose {k-1}}$ then
$\nu(H)\ge m$.
For 3-graphs, K\"{u}hn, Osthus, and Treglown \cite{KOT13} and,
independently, Khan \cite{Kh13} proved the following stronger result:
There exists  $n_0\in \mathbb{N}$ such that if $H$ is a 3-graph of order $n\geq n_0$, $m\le n/3$,
and $\delta_1(H)>{{n-1}\choose 2}-{{n-m}\choose 2},$
then $\nu(H)\ge m$.

 R\"{o}dl, Ruci\'{n}ski,  and Szemer\'{e}di \cite{RRS09}
determined the minimum codegree threshold for the
existence of a perfect matching in a $k$-graph.
Treglown and Zhao \cite{TZ12,TZ13} extended this result to include $l$-degrees with  $k/2\leq l\leq k-2$.
H\`{a}n, Person, and Schacht \cite{HPS09} considered the minimum
$l$-degree condition for perfect matchings when $1\leq l\leq k/2$. In particular, they
showed that if $H$ is a 3-graph and $\delta_1(H)>(1+o(1))\frac 5 9 {|V(H)|\choose 2}$ then $H$ has perfect matching.

For near perfect matchings, Han \cite{Han15} proved a conjecture of R\"{o}dl, Ruci\'{n}ski,  and
Szemer\'{e}di  \cite{RRS09}  that, for $n \not\equiv 0\pmod
k$, the co-degree threshold for the existence of a near perfect matching in a $k$-graph $H$
is $\lfloor n/k\rfloor$. This is much smaller than the co-degree threshold (roughly $n/2$)
obtained by R\"{o}dl, Ruci\'{n}ski,  and Szemer\'{e}di  \cite{RRS09} for perfect matchings.

For nearly perfect matchings,
H\`an, Person, and Schacht \cite{HPS09} proved the following result:
For any integers $k>l>0$ there exists $n_0\in \mathbb{N}$ such that
for all $n>n_0$ with $n\in k\Z$ and for every $n$-vertex $k$-graph $H$  with
\[
	\delta_l(H)\ge \frac{k-l}{k}\binom{n}{k-l} + k^{k+1}(\ln n)^{1/2}n^{k-l-1/2},
\]
$H$ contains a matching covering all but $(l-1)k$ vertices.
Our main result improves this bound for the range $k/2<l\le k-1$,
by providing an exact $l$-degree threshold for the existence of a matching covering all but at most $(k-l)\lceil(k-l)/(2l-k)\rceil+k-1$ vertices.

\newpage

\begin{theorem}\label{main-thm}
	For any integers $k,d$ satisfying $k\ge 3$ and $k/2<d<k-1$ (or $d=k-2$), there exists a positive real $\mu$ such that, for all integers $m,n$ satisfying
	\begin{align*}\label{thm1}
			\frac{n}{k}-\mu n\le m\le \frac{n}{k} -1- \left(1-\frac{d}{k}\right)\left\lceil\frac{k-d}{2d-k}\right\rceil
		\end{align*}
	and $n$ sufficiently large,
	if $H$ is  a $k$-graph on $n$ vertices and  $\delta_{d}(H)>{n-d\choose k-d}-{(n-d)-m\choose k-d}$
	then $\nu(H)\ge m+1$.
\end{theorem}

When $l\ge 2k/3$, we have $(k-l)/(2l-k)\le 1$. Moreover, if $n\equiv r \pmod k$ with $0\le r<k$ and $r+l\ge k$ then we can apply Theorem~\ref{main-thm} with $m=\lceil n/k\rceil -2$ and conclude that
$H$ has a matching covering all but at most $k$ vertices.
In general, if the interval $[n/k-2, n/k -1-
\left(1-l/k\right)\left\lceil(k-l)/(2l-k)\right\rceil]$ contains an
integer, then by letting $m$ to be that integer, Theorem~\ref{main-thm} implies that $H$ has a near perfect matching.

The bound on $\delta_l(H)$ in Theorem~\ref{main-thm} is best possible.
To see this, consider the  $k$-graph $H^k_k(U,W)$, where $U,W$ is a partition of $V(H^k_k(U,W))$ and
the edges of $H^k_k(U,W)$ are precisely those $k$-subsets of $V(H^k_k(U,W))$ intersecting $W$ at least once.
For integers $k,l,n$ with $k\ge 2$ and $0<l<k$ and for large $n$,
$\delta_{l}(H_k^{k}(U,W))={n-l\choose k-l}-{(n-l)-|W|\choose
  k-l}$ and the matching number of $H_k^{k}(U,W)$ is $|W|$.
Thus, the bound on $\delta_l(H)$ in Theorem~\ref{main-thm} is best possible (by letting $|W|=m$).

We need to extend the definition of
$H^k_k(U,W)$ to  $H^s_k(U,W)$ for all $s\in [k]$. Again,  $U,W$ is a
partition of $V(H^s_k(U,W))$ and  the edges of $H^s_k(U,W)$ are precisely those $k$-subsets of $V(H^s_k(U,W))$
intersecting $W$ at least once and at most $s$ times.
To prove Theorem~\ref{main-thm}, we consider two cases based on
whether or not  $H$ is ``close'' to $H_k^{k-l}(U,W)$ for some partition $U,W$ of $V(H)$.

Given two hypergraphs $H_1, H_2$ and a real number $\varepsilon > 0$,
we say that $H_2$ is \textit{$\varepsilon$-close} to $H_1$ if $V(H_1) = V(H_2)$ and $|E(H_1)\backslash E(H_2)|\leq \varepsilon |V(H_1)|^{k}$.

We first consider the case when  $V(H)$ has a partition $U,W$ with $|W|=m$ such that $H$ is
close to $H_k^{k-l}(U,W)$. If every vertex of  $H$ is ``good'' (to be made precise later) with respect to
$H_k^{k-l}(U,W)$ then we find the desired matching by a greedy
argument.
Otherwise, we find the desired matching in two steps by first
finding a matching $M'$ such that every vertex in $H-V(M')$
is good, thereby reducing the problem to the previous case.

For the case when $H$ is not close to $H_k^{k-l}(U,W)$ for any
partition $U,W$ of $V(H)$ with $W|=m$, we will see that $H$ does not have sparse set of very large size and we will use the following approach of
Alon, Frankl, Huang, R\"{o}dl, Ruci\'{n}ski, and Sudakov
\cite{AFHRRS12}:
\begin{itemize}
\item Find a small absorbing matching $M_a$ in $H$,
\item find random subgraphs of $H-V(M_a)$ with perfect fractional matchings (see Section 4 for definition),
\item use those random subgraphs and a theorem of Frankl and R\"{o}dl
  to find  an almost perfect matching $M'$ in $H-V(M_a)$ (see Lemma~\ref{almostperfect}), and
\item use the matching $M_a$ to absorb the remaining vertices in $V(H)\setminus V(M_a)\setminus V(M')$.
\end{itemize}
To find a perfect fractional matching in certain random subgraphs of $H-V(M_a)$ we need to prove a stability version of
a result of Frankl~\cite{Fr13} on the Erd\H{o}s matching
conjecture~\cite{Er65}, which might be of independent interest.
We also need to use the hypergraph container  result of Balogh, Morris, and Samotij
\cite{BMS15} to bound the independence number of a random subgraph of
$H$.

Our paper is organized as follows.
In Section 2, we  prove Theorem~\ref{main-thm} for $k$-graphs $H$ that are
$\varepsilon$-close to $H^{k-l}_k(U,W)$ (for any $\varepsilon<  (8^{k-1}k^{5(k-1)}k!)^{-3}$). In fact, in this case, we prove a more general
result about degree threshold for matchings of any size less than $ n/k$.
In Section 3, we prove an absorbing lemma that ensures the existence of a small matching $M_a$ in $H$ such
  that for any small set $S$, the subgraph of $H$ induced by $V(M_a)\cup S$ has a nearly perfect
  matching.
This is done by a standard second moment method.  In Section 4, we
show that if a $k$-graph does not have very large independence number but  has large
minimum $l$-degree then it has a perfect fractional   matching.
This is done  by proving a stability version of
a result of Frankl. In Section 5, we will see that we can control the independence
number of a $H$ when $H$ is not close to $H_k^{k-l}(U,W)$ for any
partition $U,W$ of $V(H)$ with $|W|=m$
(see Lemma~\ref{4.1}), which also allows us to apply the
hypergraph container result to control the independence number of
random subgraphs of $H-V(M_a)$.
 We find random subgraphs of $H-V(M_a)$ with
perfect fractional  matchings  and use them to find an almost perfect matching in $H-V(M_a)$, using the approach in
  \cite{AFHRRS12} which reduces the problem to a result of Frankl and
  R\"{o}dl \cite{FR85} (see Lemma~\ref{Rodl}).
In Section 6, we complete the proof of Theorem~\ref{main-thm} by applying the absorbing lemma from Section 3 to
  convert the almost perfect matching to the desired matching.
We also show how our proof implies a result on perfect matchings in 3-graphs proved by  K\"{u}hn, Osthus, and Treglown \cite{KOT13} and,
independently, Khan \cite{Kh13}.

We end this section with additional notation and terminology. For any
positive integer $n$, let $[n]:=\{1, \ldots, n\}$.  Let $H$ be a hypergraph. For $S\subseteq V(H)$, we use $H-S$ to denote the hypergraph
obtained from $H$ by deleting $S$ and all edges of $H$ with a vertex in $S$, and we use $H[S]$ to denote the hypergraph with vertex set $S$ and edge set $\{e\in E(H) : e\subseteq S\}$.
For $S\subseteq R\subseteq V(H)$, let $N_{H-R}(S)=\{T\subseteq V(H)\setminus R  : S\cup T\in E(H)\}$, and let $N_H(S):=N_{H-S}(S)$.

\section{Hypergraphs close to $H_k^{k-l}(U,W)$}

In this section, we prove Theorem \ref{main-thm} for the case when $H$ is
close to $H_k^{k-l}(U,W)$ for some partition $U,W$ of $V(H)$ with $|W|=m$.
Actually, in this case we prove a more general statement by
considering the existence of an $(m+1)$-matching in $H$ for all  $m\le
n/k  -1$.
However, in the case when $m\le n/(2k^4)$, we do not require $H$ to be
close to $H_k^{k-l}(U,W)$ or $l>k/2$.

\begin{lemma}\label{small-matching}
	Let $n,m,k, l$ be positive integers such that $k\ge 3$, $m\le n/(2k^4)$, and $ l\in [k-1]$.
	Let $H$ be a $k$-graph on $n$ vertices  and  $\delta_l(H)>
        {n-l\choose k-l}-{(n-l)-m\choose k-l}$. Then $\nu(H) \ge m+1$
\end{lemma}

\pf
	We apply induction on $m$.
	When $m=0$, we have $\delta_l(H)>0$; so $\nu(H)\geq 1$.
	Now assume $m\ge 1$ and  that the assertion  holds when $m$ is
        replaced with $m-1$.
	Let $M$ be a maximum matching in $H$, and assume $|M| \le m$.
	
	Since $M$ is a maximum matching in $H$,  every edge of $H$  intersects $M$.
	So there exists a vertex $v\in V(M)$ such that
	\begin{align*}
		d_H(v)>\frac{e(H)}{km}.
	\end{align*}	
	Note that $e(H)\ge\delta_l(H){n\choose l}/{k\choose l}$, and
	\begin{align*}
		\delta_l(H)
		&> {n-l\choose k-l}-{(n-l)-m\choose k-l} \quad \mbox{(by assumption)}\\
		&={n-l\choose k-l}\left(1-\prod_{i=0}^{m-1}\frac{n-k-i}{n-l-i}\right)\\
		&>{n-l\choose k-l}\left(1-\left(1-\frac{k-l}{n-l}\right)^{m}\right) \\
		&>{n-l\choose k-l}\left(1-\left(1-m \frac{k-l}{n-l} +\binom{m}{2} \left(\frac{k-l}{n-l}\right)^2\right)\right) \\
		&>\frac{m(k-l)}{2(n-l)}{n-l\choose k-l},
	\end{align*}
	where the last inequality holds because $m\le n/(2k^4)$.
	Thus we have
	$$d_H(v)>\frac{e(H)}{km}\ge \frac{\delta_l(H){n\choose l}}{km{k\choose l}}
		>\frac{(k-l)}{2nk}{n-l\choose k-l}\frac{{n\choose l}}{{k\choose l}}
		\geq\frac{1}{2k^2}{n-1\choose k-1}.$$

		Note that
	\begin{align*}
		\delta_l(H-v)
		& \ge \delta_l(H)-\binom{n-(l+1)}{k-(l+1)} \\
                & > {n-l\choose k-l}-{(n-l)-m\choose k-l}-\binom{n-(l+1)}{k-(l+1)} \\
		& = {(n-1)-l\choose k-l}-{((n-1)-l)-(m-1)\choose k-l}.
	\end{align*}
	Hence, by induction hypothesis, $H-v$ has a matching of size
	$m$, say $M'$.

	The number of edges of $H$ containing $v$ and intersecting
        $V(M')$ is at most $	km{n-2\choose k-2}$. Since $n\geq 2k^4m$,
	\begin{align*}
		km{n-2\choose k-2}<\frac{1}{2k^2}{n-1\choose k-1}<d_H(v).
	\end{align*}
	Thus $H-V(M')$ contains an edge $e$ such that $v\in e$. Now $M'\cup \{e\}$
	is a matching in $H$ of size $m+1$.
\qed

\medskip

For the case when $m>n/(2k^4)$, we use the structure of  $H_k^{k-l}(U,W)$ to help us
construct the desired matching in $H$. First, we prove a lemma for the  case where, for each
vertex $v\in V(H)$, only a small number of edges of $H_k^{k-l}(U,W)$ containing $v$ do  not belong to $H$.

Let $H$ be a $k$-graph and let $U,W$ be a partition of $V(H)$ and let $n=|U|+|W|$. Given
real number $\alpha$ with $0<\alpha <1$, a vertex $v\in V(H)$ is called
\emph{$\alpha$-good} with respect to $H_k^{k-l}(U,W)$ if
$$\left|N_{H_k^{k-l}(U,W)}(v)\setminus N_H(v)\right|\le \alpha n^{k-1}; $$
and, otherwise, $v$ is called
\emph{$\alpha$-bad}. This notion quantifies the closeness of $H$ to $H_k^{k-l}(U,W)$ at a vertex.
Clearly, if $H$ is $\varepsilon$-close to $H_k^{k-l}(U,W)$, then the
number of $\alpha$-bad vertices in $H$ is at most $k\varepsilon n/\alpha$.

Note that in the statement of the lemma below we use $m\ge n/(2k^5)$ for its application in the proof of Lemma~\ref{Phk},
rather than $m\ge n/(2k^4)$ as opposed to Lemma~\ref{small-matching}.

\begin{lemma}\label{good}
	Let $k,l, m,n$ be integers and $\alpha$ be a positive real, such that $k\ge 3$, $l\in [k-1]$,
         $\alpha<(8^{k-1}k^{5(k-1)}k!)^{-1}$,  $n\ge 8k^6$, and $n/(2k^5) < m \le  n/k$.
	Suppose that $H$ is a $k$-graph and $U,W$ is a partition of
        $V(H)$ such that $|W|=m$, $|U|=n-m$, and
	every vertex of $H$ is $\alpha$-good with respect to $H_k^{k-l}(U,W)$.
	Then $\nu (H)\ge m$.
\end{lemma}

\pf
	We find a matching of size $m$ in $H$
	using those edges intersecting $W$ just once.
	Let $M$
	be a maximum matching in $H$
	such that $|e\cap W|=1$ for each $e\in M$, and let $t=|M|$.
	We may assume $t<m$; or else the desired matching exists.  So $W\setminus V(M)\ne \emptyset$.
	 By the maximality of
	$M$, $N_H(x)\cap {U\setminus V(M)\choose k-1}=\emptyset$ for all $x\in W\setminus V(M)$.

	We claim that $t\ge m/2$.
	For, suppose $t< m/2$. Since
	$m\le n/k$, $t<n/(2k)$; so $|V(H)\setminus
        V(M)|=n-tk>n-n/2=n/2$. Hence,
       $$|U\setminus V(M)|>|V(H)\setminus V(M)|-|W|\ge n/2-n/k\ge n/6.$$
       Thus, for any $x\in W\setminus V(M)$,
	\[
		\left| N_{H_{k}^{k-l}(U,W)}(x) \setminus N_H(x)\right|\geq \bigg| {|U\setminus V(M)|\choose k-1} \bigg| > {n/6\choose k-1}>\alpha n^{k-1},
	\]
	contradicting the assumption that $x$ is $\alpha$-good.

 	Since $t<m\le n/k$ and $|e\cap W|=1$ for each $e\in M$, there exists a $k$-set
  	$S=\{u_1,\ldots,u_k\}\subseteq V(H)\setminus V(M)$ such that $u_k\in W$ and
  	$S\setminus \{u_k\}\subseteq U$.
	Since $m\ge n/(2k^5)> 2k$, we have $t\ge m/2> k$.

	Arbitrarily choose $k-1$ pairwise distinct edges $e_1,\dots ,e_{k-1}$ from $M$
	and write $e_i := \{v_{i,1},v_{i,2},\ldots,v_{i,k}\}$ such that
	$v_{i,k} \in W$, and $v_{i,j} \in U$ for $j \in [k-1]$.
	For convenience, let $v_{k,j} := u_j$ for $j \in [k]$.
	For $i \in [k]$, define $f_i := \{v_{1,1+i},
	v_{2,2+i},\ldots,v_{k-1,(k-1)+i},v_{k,k+i}\}$, where the addition in the
	subscripts is modulo $k$ (except that we write $k$ for $0$).
	Then $f_i \not\in E(H)$ for some $i \in [k]$
	as, otherwise, $(M\setminus \{e_i\ :\ i\in [k-1]\})\cup \{f_i\ :\ i\in [k]\}$ is
	a matching in $H$ that contradicts the maximality of $M$.
	
        Note that for different choices of $e_1,\ldots,e_{k-1}\in M$ and $e_1',\ldots,e_{k-1}'\in M$,
        the corresponding sets  $\{f_1,\ldots, f_k\}$ and $\{f_1',\ldots, f_k'\}$ constructed in the above paragraph are disjoint.
 	Since there are ${t\choose k-1}$ choices of  $e_1,\ldots,e_{k-1}$ from $M$, we have
 	\begin{eqnarray*}
 		&& \sum_{i=1}^{k} \left|N_{H^{k-l}_k(U,W)}(u_i)\setminus N_H(u_i)\right| \\
		&\ge& {t\choose k-1} \\
		&>& \frac{(t-(k-1)+1)^{k-1}}{(k-1)!} \\
		&>&\frac{(n/(4k^5)-(k-1))^{k-1}}{(k-1)!} \quad \quad \text{(since } t\ge m/2> n/(4k^5) \text{)}\\
		&>& \frac{(n/(8k^5))^{k-1}}{(k-1)!} \quad \quad \text{(since } n\ge 8k^6 \text{)}\\
		&=& (8^{k-1}k^{5(k-1)}k!)^{-1} k n^{k-1} \\
		&>&\alpha kn^{k-1} \quad \quad \text{(since } \alpha<(8^{k-1}k^{5(k-1)}k!)^{-1} \text{)}.
 	\end{eqnarray*}
 	Thus there exists $u_j\in S$ such that
 	\[
 		\left|N_{H^{k-l}_k(U,W)}(u_j)\setminus N_H(u_j)\right|
		>\alpha n^{k-1},
 	\]
	contradicting the assumption that $u_j$ is $\alpha$-good with
        respect to $H_k^{k-l}(U,W)$.
\qed

\medskip

The next lemma takes care of Theorem~\ref{main-thm} for the case when
$m>n/(2k^4)$ and $H$ is $\varepsilon$-close to $H_k^{k-l}(U,W)$. We treat those vertices that are not
$\sqrt{\varepsilon}$-good separately by finding two matchings
covering those vertices (in
two steps, and using Lemma~\ref{small-matching}), and then, Lemma~\ref{good} will be
used to find a matching covering the remaining vertices.

\begin{lemma}\label{Phk}
	Let $k, l, m, n$ be integers and let
        $0<\varepsilon<  (8^{k-1}k^{5(k-1)}k!)^{-3}$, such that
	$k\ge 3$, $l\in [k-1]$, $n\ge 8k^6/(1-5k^2\sqrt{\varepsilon})$, and
	$n/(2k^4)<m\le n/k  -1$.  Suppose $H$ is a  $k$-graph on $n$
        vertices such that  $\delta_l(H)> {n-l\choose k-l}-{n-l-m \choose k-l}$
	and $H$ is $\varepsilon$-close to $H_k^{k-l}(U,W)$, where $|W|=m$ and $|U|=n-m$.
	Then $\nu(H)\ge m+1$, or $l=k-1$, $n\in k\Z$, $m=n/k-1$, and $\nu(H) \ge m$.
\end{lemma}

\pf
	Since $H$ is $\varepsilon$-close to $H_k^{k-l}(U,W)$,
	all but at most $k\sqrt{\varepsilon}n$ vertices of $H$
	are $\sqrt{\varepsilon}$-good with respect to $H_k^{k-l}(U,W)$.
	Let $U^{bad}$ (respectively, $W^{bad}$)  denote the set of
        $\sqrt{\varepsilon}$-bad vertices in $U$ (respectively, $W$) with respect to $H_k^{k-l}(U,W)$.
	So $|U^{bad}|+|W^{bad}|\leq k\sqrt{\varepsilon}n$.
	Let $c:=|W^{bad}|$, $V_1:=U\cup W^{bad}$, and $W_1:=W\setminus
        W^{bad}$.
	Note that possibly $c=0$ and $W^{bad}=\emptyset$.
	We deal with vertices in $W_1$ later since at those vertices  $H$ and
        $H_k^{k-l}(U,W)$ are ``close''.
       We claim that

     \begin{itemize}
         \item [(1)] $H[V_1]$ has a matching $M_1$ of size $c+1$.
     \end{itemize}
        To see this, let $s$ be the maximum number of edges in $H$ intersecting $W_1$
	and containing a fixed $l$-set in $V_1$. Then $s\le {n-l\choose k-l}-{n-l-(m-c)\choose k-l}$ and
	 $\delta_l(H[V_1])\ge \delta_l(H)-s$. Hence,
	\begin{align*}
		\delta_l(H[V_1]) \ge \delta_l(H)-s > {(n-m+c)-l\choose k-l}-{(n-m+c)-l-c\choose k-l}.
	\end{align*}
        Since $n/(2k^4)<m<n/k\le n/3$,  we have $n-m+c>2m+c>n/k^4+c$. Thus, since $c\le k\sqrt{\varepsilon} n$, $n-m+c>2k^4c$ by the choice of $\varepsilon$.	
 	So by Lemma~\ref{small-matching}, $H[V_1]$ contains a matching of size $c+1$. $\Box$

\medskip

	Let $H_1:=H-V(M_1)$. Next, we cover $U^{bad}\cup W^{bad}$ with two matchings in $H_1$, which use edges intersecting $W_1$ at most once.
       First note that, for each $l$-set $S\subseteq V_1\setminus V(M_1)$, $H_1$ has lots of edges  containing $S$ and intersecting $W_1$ just once, or
       $H_1$ has lots of edges of containing $S$ and contained in $V_1\setminus V(M_1)$. More precisely, we show that
       \begin{itemize}
        \item [(2)]   for any real number $\beta$ with $2k^2\sqrt{\varepsilon}<\beta<(2k)^{-(k-l+3)}/2-k^2\sqrt{\varepsilon}$ (which exists as $\varepsilon<(2k)^{-2k-11}$ and $k\ge 3$) and
         for any $S\in{ V_1\setminus V(M_1)\choose l}$,
	\[
		\left|\{T\in N_{H_1}(S)\ :\ |T\cap W_1|=1\} \right|\geq \beta n^{k-l}, \quad \text{or}
	\]
	\[
		\left|\{T\in N_{H_1}(S)\ :\ T\subseteq V_1\setminus V(M_1)\}\right|\ge \beta n^{k-l}.
	\]
       \end{itemize}
	For, suppose  $S\in{ V_1\setminus V(M_1)\choose l}$ and $|\{T\in N_{H_1}(S)\ :\ |T\cap W_1|=1\} |< \beta n^{k-l}$.
	Since
	\[
		|\{T\in N_{H_1}(S)\ :\ |T\cap W_1|\ge 2\} |
		\le \sum_{i=2}^{k-l}\binom{m}{i}\binom{n-l-m}{k-l-i}
	\]
	and
         $$| \{T\in N_{H}(S)\ :\ |T\cap V(M_1)|\ge 1\}|\le
        k(c+1)n^{k-l-1} <2k^2\sqrt{\varepsilon}n^{k-l},$$ we have
	\begin{eqnarray*}
		&& \ |\{T\in N_{H_1}(S)\ :\ T\subseteq V_1\setminus V(M_1)\}| \\
		&> &\delta_l(H)-|\{T\in N_{H_1}(S)\ :\ |T\cap W_1|\ge 2\} | \\
		  &&\quad \quad -|\{T\in N_{H_1}(S)\ :\ |T\cap W_1|=1\} |-2k^2\sqrt{\varepsilon}n^{k-l}\\
                &> &\left( \binom{n-l}{k-l} - \binom{n-l-m}{k-l}\right)- \sum_{i=2}^{k-l}\binom{m}{i}\binom{n-l-m}{k-l-i}
			               -\beta n^{k-l} 		
        	       -2k^2\sqrt{\varepsilon}n^{k-l}  	
                   \\
		&= &   m\binom{n-l-m}{k-l-1} -2k^2\sqrt{\varepsilon}n^{k-l} -\beta n^{k-l} \\
                & > & n^{k-l}/(2k)^{k-l+3}	-2k^2\sqrt{\varepsilon}n^{k-l} -\beta n^{k-l}   \quad \mbox{(since $n/(2k^4)\le m<n/k$ and $n\ge 8k^6$)}\\
        	&\geq & \beta n^{k-l}  \quad \mbox{ (by the choice of
                        $\beta$)}. \quad \Box
	\end{eqnarray*}

        To find matchings in $H_1$ covering $(U^{bad}\cup W^{bad})\setminus V(M_1)$,
	we fix a set $B\subseteq V_1\setminus V(M_1)$ such that $|B|\equiv 0\pmod l$,
	$(U^{bad}\cup W^{bad})\setminus V(M_1)\subseteq B$, and
	$|B\setminus (U^{bad}\cup W^{bad})|< l$.
	For convenience, let $q=|B|/l$. Then
                $$q\leq k\sqrt{\varepsilon}n.$$
	We partition $B$ into $q$ disjoint $l$-sets $B_1,\ldots, B_q$
        and we may assume, by (2),  that,
        for some $q_1\in [q]\cup \{0\}$,  $|\{T\in N_{H_1}(B_i)\ :\ |T\cap W_1|=1\} |\geq \beta n^{k-l}$
	for $1\leq i\leq q_1$
	and $|\{T\in N_{H_1}(B_j)\ :\ T\subseteq V_1\setminus V(M_1)\} | \ge \beta n^{k-l}$
	for $q_1< j\leq q$. We claim that

     \begin{itemize}
        \item [(3)] there exist disjoint matchings $M_{21}$ and $M_{22}$ in $H_1$ such
	that $|M_{21}|+|M_{22}|\le k\sqrt \varepsilon n$,
	$M_{21}$ covers $\bigcup_{i=1}^{q_1}B_i$ and each edge in $M_{21}$ intersects $W_1$ just once, and
    $M_{22}$  covers $\bigcup_{i=q_1+1}^qB_i$ and each edge in $M_{22}$ is  disjoint from $W_1$.
      \end{itemize}
         	First, we find the matching $M_{21}$ covering $\bigcup{i=1}^{q_1}B_i$ (which is empty if $q_1=0$).
      Suppose for some $0\le h <q_1$ we have chosen pairwise disjoint edges $e_1,\ldots,e_h$ of $H_1=H-V(M_1)$ (which is empty when $h=0$),
        such that, for $i \in [h]$, we have
	$|e_i\cap W|=1$ and  $B_i\subseteq e_i$.
	 Since $|\{T\in N_{H_1}(B_{h+1}) : |T\cap W_1|=1\} |\geq \beta n^{k-l}$,
	 the number of edges of $H$ disjoint from
	$V(M_1)\cup \left(\bigcup_{i=1}^h e_i\right)$ but
	containing $B_{h+1}$ and exactly one vertex from  $W_1$  is at least
	\[
		\beta n^{k-l}-k|M_1|n^{k-l-1}-(hk)n^{k-l-1} \ge \beta n^{k-l}-2k^2\sqrt{\varepsilon}n^{k-l}>0,
	\]
	since $h\le q_1-1\le k\sqrt{\varepsilon} n-1$. Thus, there is an  edge $e_{h+1}$  of $H_1$ such that
	$|e_{h+1}\cap W_1|=1$, $B_{h+1}\subseteq e_{h+1}$, and $e_{h+1}\cap
	\left(\bigcup_{j=1}^{h}e_j\right)=\emptyset$. Since $q_1\le q\le k\sqrt{\varepsilon}n$, we may continue this process till $h=q_1-1$.
        Now $M_{21}=\{e_1,\ldots,e_{q_1}\}$ is the desired matching
        that covers $\bigcup_{i=1}^{q_1}B_i$.

	Next,  we find the matching $M_{22} =\{e_j : q_1< j\leq q\}$,
	such that for $q_1< j\leq q$, $B_j\subseteq e_j$ and $e_j\subseteq V_1\setminus V(M_1)\setminus \left(\bigcup_{s=1}^{j-1}e_s\right)$.
	Suppose that we have chosen $e_1, \ldots, e_{q_1},\ldots, e_s$ for some $s$ with
	$q_1\le s<q$ (which is empty if $q_1=q$).
	Since $|\{T\in N_{H_1}(B_{s+1}) : T\subseteq V_1\setminus V(M_1)\} | \ge \beta n^{k-l}$,
	the number of edges in $H$ disjoint from  $V(M_1)\cup \left(\bigcup_{i=1}^s e_i\right)\cup W_1$
	but containing $B_{s+1}$
	is at least
	\[
		\beta n^{k-l}-k|M_1|n^{k-l-1}-(sk)n^{k-l-1}
		\ge \beta n^{k-l}-2k^2\sqrt{\varepsilon}n^{k-l}>0,
	\]
        since $s\le q-1\le k\sqrt{\varepsilon} n-1$.
	So there exists an edge $e_{s+1}$ of $H_1$ such that
        $B_{s+1}\subseteq e_{s+1}$ and
	$e_{s+1}\cap \left(\bigcup_{i=1}^s e_i\right)=\emptyset$. Since $q\le k\sqrt{\varepsilon}n$,
        we may continue this process till $s=q-1$. Now
        $M_{22}=\{e_{q_1+1}, \ldots, e_q\}$ gives the desired matching
        that covers $\bigcup_{i=q_1+1}^q B_i$. \quad $\Box$

\medskip

	Now, every vertex in $V(H) \setminus V(M_1\cup M_{21}\cup M_{22})$ (as a vertex of $H$) is
	$\sqrt{\varepsilon}$-good with respect to $H_k^{k-l}(U,W)$. 	
	In order to apply Lemma~\ref{good}, we find a matching $M_{23}$ in $H_1-V(M_{21}\cup M_{22})$
        such that  every vertex of $H_2:=H_1-V(M_{21}\cup M_{22}\cup M_{23})$ is $\varepsilon^{1/3}$-good with respect to $H_k^{k-l}(U^*,W^*)$,
        where $U^*=U\cap V(H_2)$ and $W^*=W\cap V(H_2)$, $|U^*|+|W^*|\ge
        8k^6$, and  $(|U^*|+|W^*|)/(2k^4)< |W^*|\le
        (|U^*|+|W^*|)/k$. So we need (4) and (5) below.

  \begin{itemize}
    \item [(4)] There exists a matching $M_{23}$ in $H_1-V(M_{21}\cup M_{22})$  such that $|M_{23}|<k\sqrt{\varepsilon}n$ and,
         if we let  $H_2:=H_1-V(M_{21}\cup M_{22}\cup M_{23})$,
                $U'=U\cap V(H_2)$, $W'=W\cap V(H_2)$, then, for some
                $r\in \{0,1\}$ with $r=0$
                for $l\le k-2$, we have
               $|W'|-r=m-c-|M_{21}|-|M_{22}|-|M_{23}|$, $|U'|+|W'|-r\ge 8k^6$, and
                 $(|U'|+|W'|-r)/(2k^5)< |W'|-r$. Moreover,   $|W'|-r\le (|U'|+|W'|-r)/k$, or
                 $l=k-1$, $r=1$,  $n\in k\Z$, $m=n/k-1$, and
                 $|W'|-r\le (|U'|+|W'|)/k$,

 \end{itemize}	
       Note $|M_1\cup M_{21}\cup M_{22}|= (c+1)+q \leq
         3k\sqrt{\varepsilon}n$ as $c,q\le
         k\sqrt{\varepsilon}n$. We consider two cases.

\medskip

   {\it Case} 1. $l\le k-2$.

       In this case, we construct the matching
         $M_{23}$ as follows.    Suppose for some
        $1\le t\le q-q_1$, we found vertices $x_1, \ldots, x_{t-1}$ in $U\setminus V(M_1\cup
        M_{21}\cup M_{22})$ and edges $f_1, \ldots, f_{t-1}$ in
        $H_1-V(M_{21}\cup M_{22})$ such that, for $i\in [t-1]$,
        we have $x_i\in f_i$, $|f_i\cap W_1|=2$,
        and $f_i\cap
        \left(\bigcup_{j=1}^{i-1}f_j\right)=\emptyset$. (When $t=1$,
        these sequences are empty.) Let
        $x_{t}\in U\setminus V(M_1\cup
        M_{21}\cup M_{22})\setminus
        \left(\bigcup_{i=1}^{t-1}f_i\right)$. Since $x_t$ is $\sqrt{\varepsilon}$-good with respect to $H_k^{k-l}(U,W)$, the number of edges of
        $H_1-V(M_{21}\cup M_{22})-\left(\bigcup_{i=1}^{t-1}f_i\right)$
         containing $x_{t}$ and exactly two vertices in $W_1$ is at least
         \[
         	{m-c-2(t-1)\choose 2}{n-m-1\choose k-3}-\sqrt{\varepsilon}n^{k-1}-(3k\sqrt{\varepsilon} n)
         	n^{k-2}-(kt) n^{k-2}>0,
         \]
         as $n/(2k^4)<m$, $c< k\sqrt{\varepsilon}n$, $t<k\sqrt{\varepsilon}n$, and  $\varepsilon<(8^{k-1}k^{5(k-1)}k!)^{-3}$.
         So there exists
         an edge
         $f_{t}$ in $H_1-V(M_{21}\cup
         M_{22})-\left(\bigcup_{i=1}^{t-1}f_i\right)$ such that
         $x_{t}\in f_{t}$ and $|f_{t}\cap W_1|=2$.
	This process works as long as  $t\le q-q_1$. Thus, we have a matching   $M_{23}=\{f_j\ :\ j\in [q-q_1]\}$ such that, for $j\in [q-q_1]$,
	$f_j\subseteq V(H_1)\setminus V(M_{21}\cup M_{22})\setminus \left(\bigcup_{i=1}^{j-1}f_i\right)$
	and $|f_j\cap W_1|=2$.

	 Let $H_2:=H_1-V(M_{21}\cup M_{22}\cup M_{23})$ and let $U'=U\cap V(H_2)$ and $W'=W\cap V(H_2)$.
	 Note that $|M_{23}|=|M_{22}|$,
        and note that
        \begin{align*}
        		& |W'|=|W|-c-|M_{21}|-2|M_{23}|=|W|-c-|M_{21}|-|M_{22}|-|M_{23}|, \mbox{ and}\\
		        & |U'|=|U|- (k(c+1)-c)-(k-1)|M_{21}|-k|M_{22}|-(k-2)|M_{23}| \\
		 & \quad \ ~ = |U|-(k-1)(c+1+|M_{21}|+|M_{22}|+|M_{23}|)-1.
	 \end{align*}	
        Hence, we have
        \begin{align*}
        	|U'|+|W'|=|U|+|W|-k(c+1)-k|M_{21}|-k|M_{22}|-k|M_{23}|.
	 \end{align*}
        Thus, $|U'|+|W'|\ge n-5k^2\sqrt{\varepsilon}n\ge 8k^6$  and, since $m\le n/k -1$,
       \begin{align*}
          	(|U'|+|W'|)/k & = (|U|+|W|)/k-(c+1)-|M_{21}|-|M_{22}|-|M_{23}| \\
			& \ge (|W|+1)-(c+1)-|M_{21}|-|M_{22}|-|M_{23}|\\
			& =|W'|.
	\end{align*}
        Moreover,  since $|W|> n/(2k^4)$ and $|W|\ge|W'|\ge |W|-
        3k\sqrt{\varepsilon} n$, we have
     \begin{eqnarray*}
          & &  (|U'|+|W'|) -2k^5|W'|\\
          & = &   |U|+|W|-k(c+1)-k|M_{21}|-k|M_{22}|-k|M_{23}|-2k^5|W'|\\
          & < & |U|+|W|-2k^5|W'| \\
          & < & 2k^4|W|-2k^5|W'|\\
         & < & 0  \quad (\mbox{since $n$ is large and $\varepsilon$ is small})
     \end{eqnarray*}

    {\it Case} 2. $l=k-1$.

         Choose arbitrary $q-q_1$ pairwise disjoint $(k-1)$-sets in $V(H)\setminus V(M_1\cup M_{21}\cup M_{22})$,
        each containing exactly two vertices in $W_1$. Note that this can be done, because $|W_1|=m-c\ge n/(2k^4)-k\sqrt{\varepsilon}n>2q$.
        Since $\delta_{k-1}(H)> m \ge n/(2k^4)>5k^2\sqrt\varepsilon n\ge  k((c+1)+3q)$, we can extend these $q-q_1$ sets
        to $q-q_1$ pairwise disjoint edges $f_1, \ldots, f_{q-q_1}$  in $H-V(M_1\cup M_{21}\cup M_{22})$.

          Clearly,  each $f_i$ contains either two or three vertices from
        $W_1$. Thus,  there exists some integer $p$ with $0\le p\le q-q_1$ such that $q-q_1+p-1\le|W_1\cap \left(\bigcup_{i=1}^{p}f_i\right)|\le q-q_1+p$.
	Let $M_{23}=\{f_1, \ldots, f_p\}$, $H_2:=H_1-V(M_{21}\cup M_{22}\cup M_{23})$, and $U'=U\cap V(H_2)$ and $W'=W\cap V(H_2)$.

          Note that $|W_1\cap V(M_{23})|=|M_{22}|+|M_{23}|-r$ for some $r\in \{0,1\}$. Hence,
         $$ |W'|= |W|-c-|M_{21}|-|W_1\cap V(M_{23})|=|W|-c-|M_{21}|-|M_{22}|-|M_{23}|+r$$ and
            $$|U'|=|U|-(k(c+1)-c)-(k-1)|M_{21}|-k|M_{22}|-(k|M_{23}|-|W_1\cap V(M_{23})|).$$ Therefore,
             $$|U'|+|W'|-r=(|U|+|W|-r)-k(c+1)-k|M_{21}|-k|M_{22}|-k|M_{23}|.$$

          It is easy to see that the same calculations in Case 1 also allow us to conclude that
        $|U'|+|W'|-r\ge 8k^6$ and
        $(|U'|+|W'|-r) -2k^5(|W'|-r)<0$. Moreover, if $r=0$ then the same argument in Case 1 shows that $|W'|\le (|U'|+|W'|)/k$. So we may assume $r=1$.

         First, suppose  $m<n/k-1$ or $n\notin k\Z$. Then $(|U|+|W|)/k\ge |W|+1+1/k$; so
       \begin{align*}
          	(|U'|+|W'|-1)/k & = (|U|+|W|-1)/k-(c+1)-|M_{21}|-|M_{22}|-|M_{23}| \\
			& \ge (|W|+1+1/k)-1/k -(c+1)-|M_{21}|-|M_{22}|-|M_{23}|\\
			& = |W'|-1.
	\end{align*}

        Now suppose  $n\in k\Z$ and $m=n/k-1$.  Then
       \begin{align*}
          	(|U'|+|W'|)/k & = (|U|+|W|)/k-(c+1)-|M_{21}|-|M_{22}|-|M_{23}| \\
			& \ge (|W|+1)-(c+1)-|M_{21}|-|M_{22}|-|M_{23}|\\
			& \ge |W'|-1. \quad \Box
	\end{align*}

\medskip

   We now define $W^*\subseteq W'$ and $U^*=V(G)\setminus W^*$ as
   follows: If $r=0$ let $W^*=W'$.  If $r=1$ and $n\notin k\Z$ or $m<n/k-1$ then
   choose some $w\in W'$ and  let $W^*=W'\setminus \{w\}$. If
   $r=1$,    $n\in k\Z$, and $m=n/k-1$ then choose $w_1,w_2\in W'$ and let
   $W^*=W'\setminus \{w_1,w_2\}$.

    \medskip
   \begin{itemize}
       \item [(5)] Every vertex of $H_2:=H_1-V(M_{21}\cup M_{22}\cup M_{23})$ is
	$\varepsilon^{1/3}$-good in $H$ with respect to $H_k^{k-l}(U^*,W^*)$.
   \end{itemize}
	Note that $k|M_1\cup M_{21}\cup M_{22}\cup M_{23}|+2\leq k((c+1)+3q)+2\leq 5k^2\sqrt{\varepsilon}n$.
	For each $x\in V(H_2)$, since $x$ is $\sqrt{\varepsilon}$-good
	with respect to  $H^{k-l}_k(U,W)$, we have
	\[
		|N_{H_k^{k-l}(U,W)}(x)\setminus N_H(x)|\leq \sqrt{\varepsilon}n^{k-1}.
	\]
	Thus,
	\begin{eqnarray*}
  		& & \left|N_{H_k^{k-l}(U^*,W^*)}(x)\setminus N_{H_2}(x)\right|\\
		&\leq &\left|N_{H_k^{k-l}(U,W)}(x)\setminus
                        N_H(x)\right|+\left(k\left|M_1\cup M_{21}\cup
                        M_{22}\cup M_{23}\right|+2\right)n^{k-2}\\
  		&\leq &\sqrt{\varepsilon}n^{k-1}+5k^2\sqrt{\varepsilon}n^{k-1}\\
  		&<&\varepsilon^{1/3}n^{k-1}.  \quad \Box
	\end{eqnarray*}

   Hence, by (4) and (5), it follows from  Lemma~\ref{good} that there is a matching $M_3$ in
	$H_2$ of size $|W^*|$.
	Let $M:=M_1\cup M_{21}\cup M_{22}\cup M_{23}\cup M_3$. Then
        $M$  is a
        matching in $H$.
         If $l=k-1$, $n\in k\Z$, and $m=n/k-1$, then $|W^*|\ge
         |W'|-r-1$; so $$|M|\ge (|W'|-r-1)+
         |M_1|+|M_{21}|+|M_{22}|+|M_{23}|= n/k-1.$$
     Otherwise, $|W^*|=|W'|-r$ and    $|M|=(|W'|-r)+ |M_1|+|M_{21}|+|M_{22}|+|M_{23}|=
               m+1.$
 \qed


\section{An absorbing Lemma}

A typical approach for finding large matchings in a dense $k$-graph $H$ is
to find a small matching $M_a$ in $H$ such that, for each small set $S\subseteq
V(H)\setminus V(M_a)$, $H[V(M_a)\cup S]$ has a large matching (e.g.,
an almost
perfect matching).
Such a matching $M_a$ is known as an {\it absorbing} matching, often found by applying the second moment method.
This approach was initiated   by R\"{o}dl, Ruci\'{n}ski,  and Szemer\'{e}di  \cite{RRS06}.

Let $Bi(n, p)$ be the binomial distribution with parameters $n$ and $p$.
The following lemma on Chernoff bound can be found in Alon and
Spencer \cite{AS08} (page 313, also see \cite{Mi05}).

\begin{lemma}[Chernoff]\label{chernoff}
Suppose $X_1,\ldots, X_n$ are independent random variables taking values in $\{0, 1\}$. Let $X=\sum_{i=1}^n X_i$ and $\mu = \mathbb{E}[X]$. Then, for any $0 < \delta \leq 1$,
$$\mathbb{P}[X \geq (1+\delta) \mu] \le  e^{-\delta^2\mu/3}\mbox{ and } \mathbb{P}[X \leq (1-\delta) \mu] \le  e^{-\delta^2 \mu/2}.$$
In particular, when $X \sim Bi(n,p)$ and $\lambda<\frac{3}{2}np$,  then
\[
\mathbb{P}(|X-np|\geq \lambda)\leq e^{-\Omega(\lambda^2/np)}.
\]
\end{lemma}

Our objective in this section is to prove a single absorbing lemma for $k$-graphs with large  $l$-degrees for the entire range $k/2<l<k$.  Thus, given
positive integers $k,l$ such that $k/2<l<k$, we consider positive integers $a,h$
satisfying $h\le l$ and $al\ge a(k-l)+(k-h)$.
Note that  when $l>k/2$, $a=k-l$ and $h=l$ always satisfy these requirements.

We will frequently use the following fact:
For integers $0\le l'<l<k$ and any $k$-graph $H$, if $\delta_l(H)\geq c{n-l\choose k-l}$ for some $0\leq c\leq 1$,
then $\delta_{l'}(H)\geq c{n-l'\choose k-l'}$.
(This is because $\delta_{l'}(H)\ge {n-l'\choose l-l'}{n-l\choose k-l}/{k-l'\choose l-l'}$.)

\begin{lemma}\label{Absorb-lem}
	Let $k, l$ be integers such that $k\ge 3$ and $k/2<l<k$, and  let $c>0$ be a constant such that $c<1/k!$. Then there exist $\rho>0$ and $c'>0$  such that
       $0< \rho\ll c'\ll c$ and the following holds for all sufficiently large integers $n$:

	Let $a,h$ be positive integers such that $h\le l$, $a\le k-l$, and $al\ge a(k-l)+(k-h)$.
	If $H$ is a $k$-graph of order $n$ and
	$\delta_{l}(H)\geq c\binom{n-l}{k-l}$,
	then there exists a matching $M$ in $H$
	such that $|M|\le 2k\rho n$ and, for any subset $S\subseteq V(H)$ with $|S|\le c'\rho n$,
     $H[V(M)\cup S]$ has a matching covering all but at most $al+h-1$ vertices.
\end{lemma}

\pf
	For $R\in \binom{V(H)}{al+h}$ and $Q\in \binom{V(H)}{ak}$,
	we say that $Q$ is \emph{$R$-absorbing}
	if  $\nu(H[Q\cup R])\ge a+1$ and $Q$ is the vertex set of a matching in $H$.
        (This requires $al+h\ge k$, guaranteed by our choice of $a$ and $h$.)
	Let $\mathcal{L}(R)$ denote the collection of all $R$-absorbing sets in $H$. We claim that

	\begin{itemize}
        \item [(1)] there exists $c'>0$ (dependent on $c,k$) such that $|\mathcal{L}(R)|\geq c'n^{ak}$ for every $R\in {V(H)\choose al+h}$.
        \end{itemize}
	To prove (1), let $R\in {V(H)\choose al+h}$
        and we wish to extend $R$ to a matching of size $a+1$ by adding a set of size $(a+1)k-(al+h)=a(k-l)+(k-h)$.	
       Partition $R$ into $a+1$ pairwise disjoint subsets $R_1, \dots , R_{a+1}$,
	with $|R_i|=l$ for $i\in [a]$, and $|R_{a+1}|=h$. Next we choose $(k-l)$-sets $T_s$ for $s\in [a]$ and a $(k-h)$-set $T_{a+1}$ such that $\{R_s\cup T_s:
       s\in [a+1]\}$ form a matching in $H$.
	
      For $j\in [a]$, since $d_H (R_{j})\ge \delta_l(H)\ge
      c\binom{n-l}{k-l}$, we have, for large $n$,
	\[
	 \left|N_{H-R-\bigcup_{s=1}^{j-1}T_s}(R_{j+1})\right| \ge 	c{n-l\choose k-l}-((al+h)+(k-l)j){(n-l)-1\choose (k-l)-1}>\frac{c}{2}{n-l\choose k-l};
	\]
       thus, we have more than $\frac{c}{2}{n-l\choose k-l}$ choices for each $T_j$ with $j\in [a]$.
	Similarly,  since $d_H(R_{a+1})\ge c\binom{n-h}{k-h}$ as $h\le l$, we have
	\[
	\left|N_{H-R-\bigcup_{s=1}^{a}T_s}(R_{a+1})\right| \ge	c{n-h\choose k-h}-((al+h)+(k-l)a){(n-h)-1\choose (k-h)-1}>\frac{c}{2}{n-h\choose k-h};
	\]
	hence, we have more than $\frac{c}{2}{n-h\choose k-h}$ choices for  $T_{a+1}$.

        Fix an arbitrary choice of $T_i\in N_{H-R-\bigcup_{s=1}^{i-1}T_s}(R_i)$ for $i\in [ a+1]$,
        and let $T=\bigcup_{i=1}^{a+1} T_i$. Next, we form an $R$-absorbing set  $Q$ by extending the set $T$ to a matching of size $a$.
        We partition $T$ into subsets $T'_1,\dots, T'_a$ such
        that $1\le |T_i'|\le l$ for $i\in [a]$. Such a partition exists since $|T|=a(k-l)+(k-h)\le al$ (by assumption).
        Similar to the arguments in the previous paragraph, we can show that there exist $P_i\in N_{H-R-T-\bigcup_{s=1}^{i-1}P_s}(T_i')$ for $i\in [a]$, such that
	\[
	 \left|N_{H-R-T-\bigcup_{s=1}^{i-1}P_s}(T_i')\right| 	> \frac{c}{2}\binom{n-|T_i'|}{k-|T_i'|}.
	\]
        This means that there are more than $\frac{c}{2}\binom{n-|T_i'|}{k-|T_i'|}$ choices for each $P_i$ with $i\in [a]$.
	Let $Q=T\cup \left(\bigcup_{i=1}^a P_i\right)$.
	Then $Q$ is the vertex set of a matching of size $a$ in $H$. Hence $Q$ is an $R$-absorbing set.

        Note that each such $ak$-set $Q$ can be produced at most $(ak)!$ times by the above process, and recall that $\sum_{i=1}^a |T_i'|=a(k-l)+(k-h)$.
	Hence, for large $n$ (compared with $k$),  we have	
	\begin{align*}
		|\mathcal{L}(R)|
		&>((ak)!)^{-1}\left(\frac{c}{2}{n-l\choose k-l}\right)^{a}
			\left(\frac{c}{2}{n-h\choose k-h}\right)
			\prod_{i=1}^a\left(\frac{c}{2}{n-|T_i'|\choose k-|T_i'|}\right)\\
		&>(2(ak)!)^{-1}\left(\frac{c}{2}\right)^{2a+1}
			\left(\frac{n^{a(k-l)}}{((k-l)!)^a}\right)
			\left(\frac{n^{k-h}}{(k-h)!}\right)
			\left(\frac{n^{ak-(a(k-l)+(k-h))}}{(ak-(a(k-l)+(k-h)))!}\right)\\
		&>c'n^{ak},
	\end{align*}
	by choosing $c'<(2(ak)!)^{-1}\left(c/2\right)^{2a+1}\left(((k-l)!)^a(k-h)!(al+h-k)!\right)^{-1}$.
	{$\Box$}

\bigskip

        Choose $\rho <c'/(2a^2k^2)$. We form  a  family $\F\subseteq {V(H)\choose ak}$
        by choosing each
	member of ${V(H)\choose ak}$ independently at random with probability
		$$p=\frac{\rho n}{{n\choose ak}}.$$
     Then
        \begin{itemize}
        \item [(2)] with probability $1/2-o(1)$, all of the following hold:
        \begin{itemize}
        \item [(2a)] $|\F|\le 2\rho n$,
        \item [(2b)] $|\mathcal{L}(R)\cap \F|\ge 2c'\rho n$, and
        \item [(2c)] $\F$ contains less than $c'\rho n$ intersecting pairs.
        \end{itemize}
        \end{itemize}
        	Clearly,  $\E(|\F|)=\rho n$ and, by (1),  $\E(|\mathcal{L}(R)\cap \F|)> c'n^{ak}p>  4c'\rho n$ (as $a\ge 1$ and $k\ge 3$).  So by Lemma~\ref{chernoff},  with
	probability $1-o(1)$,
	$$	|\F|\leq 2\rho n,$$
	and, for each fixed $(al+h)$-set $R$, with probability at least
	$1-e^{-\Omega(\rho n)}$, $\F$ satisfies
	$$	|\mathcal{L}(R)\cap \F|\geq 2c' \rho n.$$
        Hence given $n$ sufficiently large,
	it follows from  union bound that,
	 with probability $1-o(1)$,
        (2a) and (2b) hold for all $(al+h)$-sets  $R$.

  	Furthermore,  the expected number of intersecting pairs in $\F$ is at most
	\[
		{n\choose ak}{ak\choose 1}{n-1\choose ak-1}p^2= a^2 k^2\rho^2 n<c'\rho n/2.
	\]
	Thus, using Markov's inequality,
	we derive that with probability at least 1/2, $\F$ contains less than $c'\rho n$ intersecting pairs of $ak$-sets.
        Hence, by a union bound, (2a), (2b), (2c) hold with probability $1/2-o(1)$. $\Box$

\medskip

	Let $\F'$ denote the family obtained from $\F$ by deleting one $ak$-set
	from each intersecting pair of sets in  $\F$ and removing all $ak$-sets that are not the vertex set of a matching in $H$ (which are not in $\mathcal{L}(R)$
        for any $(al+h)$-set $R$).
	Then  $\F'$ consists of pairwise disjoint vertex sets of matchings of size $a$ in $H$,
	and for all $(al+h)$-sets $R$,
	\[
		|\mathcal{L}(R)\cap \F'|\geq 2c'\rho n-c'\rho n
		\ge c'\rho n.
			\]

	For each $F\in \F'$, let $M_F$ be a matching in $H$ with $V(M_F)=F$. Then $M=\bigcup_{F\in \F'}M_F$ is a
	perfect matching in $H[V(\F')]$, and $|M|\le a|\F|\le  k|\F|\le  2k\rho n$.
       	It remains to show that $M$ is the desired matching.

	Let $S$ be an arbitrary subset of $V(H)\setminus V(M)$ with $|S|\leq c'\rho n$.
        We use $M$ to absorb $(al+h)$-sets  iteratively, starting with an arbitrary
        $(al+h)$-subset of $S$.
        Let $S_0:=S$ and let $R_0\subseteq S_0$ with $|R_0|=al+h$.
        Since $|\mathcal{L}(R_0)\cap \F'|\ge c'\rho n$, we can find  $Q_0\in \F'$
        such that $H[R_0\cup Q_0]$ has a matching $M_0$ with $|M_0|=a+1$. Let
        $S_1=(S_0\cup Q_0)\setminus V(M_0)$; then $|S_1|=|S_0|-k$.

        Let $t\ge 0$ be maximal such that we have sets $S_0,\ldots,
        S_{t}$, $(al+h)$-sets $R_0, \ldots, R_t$,   disjoint $ak$-sets
       $Q_0, \ldots, Q_t$, pairwise disjoint matchings $M_0, \ldots, M_t$ in $H$, such that $|S_t|\ge al+h$ and, for $i\in [t]$,  $Q_i\in \F'$ is an $R_i$-absorbing set,
       $M_i$ is a matching in $H[R_i\cup Q_i]$,  $|M_i|=a+1$, and
       $S_{i}=(S_{i-1}\cup Q_{i-1})\setminus V(M_{i-1})$.
       So $|S_i|=|S_{i-1}|-k$ for $i\in [t]$.

       Let $S_{t+1}=(S_t\cup Q_t)\setminus V(M_t)$. Then  $|S_{t+1}|<al+h$. For otherwise, let $R_{t+1}$ be an $(al+h)$-subset of $S_{t+1}$.
       Since $|\mathcal{L}(R_{t+1})\cap \F'|\ge c'\rho n$ and
       $t+1\le |S|/k+1\leq c'\rho n-1$,
       there exists $Q_{t+1}\in \F'\setminus\{Q_0,\dots, Q_t\}$
       such that $H[R_{t+1}\cup Q_{t+1}]$ has a matching $M_{t+1}$
       with $|M_{t+1}|=a+1$.  We obtain a contradiction to the maximality of $t$.
       Thus, $M$ is the desired matching. \qed


\section{Fractional perfect matchings}

A {\it fractional matching} in a $k$-graph
  $H$ is a function $w: E\rightarrow [0,1]$ such that for any
  $v\in V(H)$, $\sum_{\{e\in E: v\in e\}}w(e)\le 1$.
A fractional matching is called {\it perfect} if $\sum_{e\in E}w(e)=|V(H)|/k$.
 Any set $I \subseteq V(H)$ that contains no edge of $H$ is called an
\emph{independent set}  in $H$. We use $\alpha(H)$ to denote the size of a largest independent set in the hypergraph $H$.

In this section, we show that for any $\varepsilon >0$ and $\rho>0$ with $\rho\ll \varepsilon$, if an $n$-vertex $k$-graph $H$ has
$\alpha(H)\le (1-1/k-\varepsilon/5)n$
and $\delta_{l}(H)> {n-l\choose k-l}-{n-l-m\choose k-l}-\rho n^{k-l}$,
  then $H$ admits a
 prefect fractional matching.
The reason for the term $-\rho
n^{k-l}$ is because we will consider a hypergraph obtained from the
hypergraph in Theorem~\ref{main-thm}
by removing the vertices of an absorbing matching (see
Lemma~\ref{Absorb-lem}).   When $\alpha(H)> (1-1/k-\varepsilon/5)n$ and
$\delta_{l}(H)> {n-l\choose k-l}-{n-l-m\choose
  k-l}-\rho n^{k-l}$,  $H$ is close to $H_k^{k-l}(U,W)$ as will be shown  in
Section 5 (see Lemma~\ref{4.1}), and, hence, can  be taken care of
by
results in Section 2.

To find a perfect fractional
 matching in a $k$-graph $H$ that is  not close to
$H_k^{k-l}(U,W)$ for any partition $U,W$ of $V(H)$ with $|U|=n-m$ and $|W|=m$,
we need to consider matchings in the ``link'' graph of an
$l$-set in a related $k$-graph, which is
a $(k-l)$-graph. This is related to the following
                  well known conjecture of Erd\H{o}s \cite{Er65} on matchings in
                  uniform hypegraphs: If $F$ is a $k$-graph on $n$ vertices and $\nu(F)=s$, then
                  $e(F)\le \max\left\{{n\choose k}-{n-s\choose k},
                  {ks+1\choose k}\right\}$.
                  Frankl \cite{Fr13} proved
                  that if $n\geq (2s+1)k-s$ then $e(F)\leq {n\choose k}-{n-s\choose k}$
                 with equality if and only if $H$ is isomorphic to $H_k^k(U,W)$, with $|W|=s$ and $|U|=n-s$.
                Very recently, Frankl and Kupavskii \cite{FK18} further improved the lower bound to $n\ge (5k/3-2/3)s$ for large $s$.

                  For our purpose, we need a stability version of Frankl's  result
                  mentioned above in order to deal with the case when $e(H)>{n\choose k}-{n-s\choose k}-\rho n^k$ (due to the removal of the vertex set of an absorbing matching).

Ellis, Keller,
                  and Lifshitz \cite{EKL18}  recently proved the following
                  stability version of Frankl's result, which we state as follows using our notation:
For any $s\in \mathbb{N}$, $\eta>0$, and $\varepsilon >0$, there exists  $\delta =\delta(s,\eta,\varepsilon)>0$ such that the following holds.
Let $n,k\in \mathbb{N}$ with $k\le (\frac{1}{2s+1}-\eta)n$. Suppose $H\subseteq {[n]\choose k}$ with $\nu(H)\le s$ and
$e(H)\ge {n\choose k}-{n-s\choose k}-\delta{n-s\choose k-1}$. Then there exists $W\in {[n]\choose s}$  such that
$|E(H)\setminus E(H_{k}^{k}(U,W))| <\varepsilon{n-s\choose k}$.

The lower bound on $e(H)$ in the above result of Ellis, Keller,
                  and Lifshitz is too large for our
                  purpose.
                  However, using LP duality we only need to consider ``stable''
                  hypergraphs and for such hypergraphs we can prove a better bound on $e(H)$.

For subsets $e=\{u_1,\ldots, u_k\}, f=\{v_1,\ldots,v_k\}\subseteq [n]$ with $u_i<u_{i+1}$ and $v_i<v_{i+1}$ for $i\in [k-1]$, we write $e\le f$
if $u_i\le v_i$ for all $i\in [k]$. A hypergraph $H\subseteq {[n]\choose k}$ is said to be {\it stable} if, for $e,f\in {[n]\choose k}$ with $e\le f$,
$f\in E(H)$ implies $e\in E(H)$.
Our proof of a stability version of Frankl's theorem for stable hypergraphs uses the same
                  ideas as in \cite{Fr13}, including the following result from \cite{Fr13} which  is an extension of Katona's
Intersection Shadow Theorem \cite{Ka64}.

\begin{lemma}\label{katona}
	Let $\F\subseteq \binom{[n]}{k}$ with $\nu(\F)=s$. Then $s|\partial
        \F|\ge |\F|$, where $\partial \F$ is the shadow of $\F$, defined by
	\[
		\partial \F := \left\{G\in \binom{[n]}{k-1}:  G\subseteq F \mbox{ for some } F\in \F \right\}.
	\]
\end{lemma}

We can now state and prove the following stability version of Frankl's
result on matchings for stable hypergraphs. Note that we allow $k=1$.
\begin{lemma}\label{stafrankl}
	Let
	$k$ be a positive integer, and let $c$ and $\xi$ be constants such that $0<c<1/(2k)$ and $0<\xi\le (1+18(k-1)!/c)^{-2}$.
	Let $n,m$ be positive integers such that $n$ is sufficiently large and
	$cn\leq m\leq n/(2k)$.
	Let $H$ be a $k$-graph with vertex set $[n]$ such that
	$H$ is stable  and $\nu(H)\leq m$.
	If $e(H)>{n\choose k}-{n-m\choose k}-\xi n^{k}$, then $H$ is $\sqrt{\xi}$-close to $H_k^k([n]\setminus [m],[m])$.
\end{lemma}

\pf	Suppose $e(H)> {n\choose k}-{n-m\choose k}-\xi n^{k}$.
When $k=1$, each edge of $H$ consists of a single vertex. In this
case, since $e(H)> m-\xi n\ge m-\sqrt \xi n$ and because $H$ is stable
and $e(H)=\nu(H)\le m$, we have that $H$ is $\sqrt \xi$-close to $H_1^1([n]\setminus [m],[m])$.

Thus, we may assume $k\ge 2$. To show that $H$ is close to $H_k^k([n]\setminus [m],[m])$, we bound $e(H-[m])$ (as edges in $H-[m]$ are not in $H_k^k([n]\setminus [m],[m])$).
	Since $H$ is stable,  the vertex $m+1$ has the maximum degree in $H-[m]$.
	So
	\[
		e(H-[m])\le \frac{(n-m)}{k} |\{e\in E(H-[m]): m+1\in e\}|.
	\]
 	Thus,
	our objective is to bound the size of $\F(\{m+1\}):=\{e\in E(H-[m]): m+1\in e\}$,
	and we use such a bound to prove that
	$H$ is $\sqrt{\xi}$-close to $H_k^k([n]\setminus [m],[m])$.
	Let
	$$\sigma =\frac{2\xi(k-1)!}{c}.$$
         First, we may assume that

	\begin{itemize}
		\item [(1)] $|\F(\{m+1\})|\ge  9k\sigma n^{k-1}$.
	\end{itemize}	
	For, suppose $|\F(\{m+1\})|<9k\sigma n^{k-1}$. Then
	\[
		e(H-[m])\le \frac{(n-m)}{k} |\F(\{m+1\})|<9\sigma n^k.
	\]
	Thus
	\begin{align*}
		& \quad \ \big|E\big(H_k^k([n]\setminus [m],[m])\big)\setminus E(H)\big| \\
		& = e\big(H_k^k([n]\setminus [m],[m])\big)
			-\big(e(H)-e(H-[m])\big) \\
		& < \left({n\choose k}-{n-m\choose k}\right)
			-\left({n\choose k}-{n-m\choose k}-\xi n^{k}-9\sigma n^k\right)\\
		& \le \xi n^{k}+9\cdot \frac{2\xi(k-1)!}{c} n^k \\
		& \le \sqrt{\xi}n^{k},
	\end{align*}
	as $\xi \le (1+18(k-1)!/c)^{-2}$.
	That is, $H$ is $\sqrt{\xi}$-close to $H_k^k([n]\setminus [m],[m])$, and the assertion of the lemma holds. $\Box$
\medskip

        We will derive a contradiction to (1). First, we extend the notation $\F(\{m+1\})$ to all $Q\subseteq [m+1]$, by letting
	\[
		\F(Q) = \{e\in E(H)\ :\ e \cap [m+1]=Q\}.
	\]
	Note that
	$
		|\F(Q)| \le \binom{n-(m+1)}{k-|Q|}=\binom{n-m-1}{k-|Q|}.
	$
        Also note that,  since $H$ is stable, $|\F(\{m+1\})|\ge |\partial \F(\emptyset)|$. So Lemma \ref{katona} gives that
	\[
		m|\F(\{m+1\})|\ge m|\partial \F(\emptyset)|\ge |\F(\emptyset)|.
	\]

        We claim that

\begin{itemize}
\item [(2)] $\left( \sum_{i=1}^{m+1}|\mathcal{F}(\{i\})|\right) + m|\F (\{m+1\}) |> m{n-m\choose k-1}\left(1-\sigma\right)$.
\end{itemize}
	To prove (2), it suffices to show  $|\mathcal{F}(\emptyset)|+\sum_{i=1}^{m+1}|\mathcal{F}(\{i\})| > m{n-m\choose k-1}\left(1-\sigma\right)$.
Note that
\[
                \sum_{Q\subseteq [m+1],|Q|\geq 2}|\mathcal{F}(Q)|\le \sum_{i=2}^k\binom{m+1}{i}\binom{n-(m+1)}{k-i} 
         \]
and

\begin{align*}
		\binom{n}{k} &= \binom{n-(m+1)}{k}+(m+1)\binom{n-(m+1)}{k-1}+\sum_{i=2}^k\binom{m+1}{i}\binom{n-(m+1)}{k-i}\\
                             &=\binom{n-m}{k}+m\binom{n-(m+1)}{k-1}+\sum_{i=2}^k\binom{m+1}{i}\binom{n-(m+1)}{k-i}.
\end{align*}
Thus,
	\begin{eqnarray*}
		& &	|\mathcal{F}(\emptyset)|+\sum_{i=1}^{m+1}|\mathcal{F}(\{i\})|\\
		&=& e(H)-\sum_{Q\subseteq [m+1],|Q|\geq 2}|\mathcal{F}(Q)|\\
		&>& {n\choose k }-{n-m\choose k}-\xi n^{k}-\sum_{i=2}^k\binom{m+1}{i}\binom{n-(m+1)}{k-i} \\
		&=&  m{n-(m+1)\choose k-1}-\xi n^{k}\\
		&>& m{n-m\choose k-1}\left(1-\sigma\right)
			\quad \mbox{(since $cn\leq m\leq n/(2k)$ and
                    $n$ large)}. \quad \Box
	\end{eqnarray*}
	
	Let $t=\lceil(2+1/k)m\rceil$. Since $n\ge 2km$ and $m\ge cn$ (where $n$ is sufficiently large),  	
	$$n-(m+1)\ge 2km-(m+1)=(2+1/(k-1))m(k-1)-1>t(k-1).$$
	Let $\M=\{f_1,\dots ,f_t\}$ be $t$ pairwise disjoint $(k-1)$-subsets of $[n]\setminus [m+1]$
chosen uniformly at random.  Let $\F_i:=\{e\setminus \{i\}:e\in\F(\{i\})\}$ for $i\in [m+1]$.  Then  $\F_{m+1}\subseteq \F_m\subseteq \ldots
\subseteq \F_1$ (since $H$ is stable) and, for each fixed pair $i,j$,
	\[
		\P(f_j\in \F_{i})=\frac{|\F_{i}|}{\binom{n-(m+1)}{k-1}}.
	\]
        Let
	\begin{align*}
		x_i=
		\begin{cases}
			1, & f_i\in \F_{m+1}, \\
			0, & f_i\not\in \F_{m+1},
		\end{cases}
	\end{align*}
	and let $p=\P(x_i=1)$ (which is the same for $i\in [t]$). Now $|\F_{m+1}|=p\binom{n-(m+1)}{k-1}$. So by (1), we have

\begin{itemize}
\item [(3)] $p> 9k\sigma$.
\end{itemize}

Thus, it suffices to derive a contradiction to (3). Note that

\begin{itemize}
\item [(4)] for $1\le i<j\le t$, $\P(x_ix_j=1)\le \left(1+\frac{1}{4k}\right)p^2$.
\end{itemize}
This is because
 	\begin{align*}
		\P(x_ix_j=1)
		& = \P(x_j=1|x_i=1)\P(x_i=1) \\
		& \le \frac{|\F_{m+1}|}{\binom{n-(m+1)-(k-1)}{k-1}}\frac{|\F_{m+1}|}{\binom{n-(m+1)}{k-1}} \\
                & =\frac{\binom{n-(m+1)}{k-1}}{\binom{n-(m+1)-(k-1)}{k-1}} \cdot p^2\\
		& \le \left(1+\frac{1}{4k}\right)p^2,
	\end{align*}
as $n-(m+1)\ge (1-1/(2k))n-1$ and  $n$ is large. $\Box$

\medskip

	Define a bipartite graph $G$ with partition sets $\M$ and $\{\F_1,\dots, \F_{m+1}\}$, where $f_j\in \M$ is adjacent to $\F_i$ if, and only if, $f_j\in\F_i$.
Note that a matching of size $m+1$ in $G$ gives rise to a matching of size $m+1$ in $H$. Thus, $\nu(G)\le m$.
So by a theorem of K\"onig, $G$ has a vertex cover of size $m$, say $T$.
	Let $x=|T\cap \M|$; then $|T\cap \{\F_1,\dots, \F_{m+1}\}|=m-x$.	
        Since $\F_{m+1}\subseteq \F_m\subseteq \cdots \subseteq \F_1$, $d_G(f_j)= m+1$
	for  $f_j\in \F_{m+1}$; so $f_j\in T$ for all  $f_j\in
        \F_{m+1}$. Hence $0\le b\le x \le m$, where $b:=|\M\cap \F_{m+1}|
=\sum_{i=1}^t x_i$. So
	$
		pt=\E (b)\le m\le t/(2+1/k).
	$
	This implies
\begin{itemize}
\item [(5)]  $p\le 1/(2+1/k)<1/2$.	
\end{itemize}

Moreover,
	\[
	 	\sum_{i=1}^{m+1}|\M\cap \F_i|=e(G)\le t(m-x)+x((m+1)-(m-x))=x^2-(t-1)x+mt.
	\]
    Thus, letting $h(x,b):=x^2-(t-1)x+mt+mb$, we have
	\begin{align*}
		\E (h(x,b))
		& \ge \E\left(m|\M\cap \F_{m+1}|+\sum_{i=1}^{m+1} |\M\cap \F_i|\right)\\
		& =mt \frac{|\F_{m+1}|}{\binom{n-(m+1)}{k-1}} + \sum_{i=1}^{m+1}  t \frac{|\F_i|}{\binom{n-(m+1)}{k-1}} \\
		& =\frac{t}{\binom{n-(m+1)}{k-1}}\left( m|\F(\{m+1\})| + \sum_{i=1}^{m+1}|\mathcal{F}(\{i\})| \right)\\
		& > mt(1-\sigma)\quad \mbox{(by (2))}.
	\end{align*}

	Next we obtain an upper bound on $\E (h(x,b))$.
	Using the convexity of $h(x,b)$ (as a function of $x$ over the interval $[b,m]$)
	and the fact that $h(b,b)-h(m,b)=(t-1-m-b)(m-b)\ge 0$, we have 	
	\begin{align*}
		h(x,b) \le \max\{h(b,b), h(m,b)\} = h(b,b)=b^2-(t-1)b+mt+mb.
	\end{align*}
        Thus,
	\begin{align*}
		\E (h(x,b))
		& \le \E(b^2-(t-1)b+mt+mb)  \\
		& = \E\left(\left(\sum_{i=1}^t x_i\right)^2
			-(t-1-m)\left(\sum_{i=1}^t x_i\right)+mt\right) \\
		& \le \left(1+\frac{1}{4k}\right)p^2(t^2-t)+pt-(t-1-m)pt+mt \quad \mbox{(by (4))}.
	\end{align*}
	Hence, combining the above bounds on $\E (h(x,b))$, we have
	\[
		\left(1+\frac{1}{4k}\right)p^2(t^2-t)+pt-(t-1-m)pt+mt>  mt(1-\sigma).
	\]
	Thus,
	\begin{align*}
		\sigma mt
		& > pt \left(t-m-\left(1+\frac{1}{4k}\right)pt-2+\left(1+\frac{1}{4k}\right)p\right) \\
		& > pt\left(\left(1-\left(1+\frac{1}{4k}\right)p\right)t-m-2\right) \\
		& \ge pt\left(\left(\left(1-\frac{1}{2}\left(1+\frac{1}{4k}\right)\right)\left(2+\frac{1}{k}\right)-1\right) m- 2\right) \quad \mbox{(by (5) and the definition of $t$)}\\
		& = pt\left(\frac{2k-1}{8k^2}m- 2\right) \\
		& > ptm/(9k)  \quad \mbox{(since $m\ge cn$ and $n$ is large)}.
	\end{align*}
Therefore,  $p<9k\sigma$, contradicting (3).
Hence $H$ must be $\sqrt{\xi}$-close to $H_k^k([n]\setminus [m],[m])$.
\qed

\medskip

{\it Remark}. In the proof of Lemma~\ref{stafrankl} we require $m\le
n/(2k)$ (e.g., when we define $t$ and $\M$ before (3)). We will see in
Section 6, we can replace it with $n/2-1$ when $k=3$ and $l=1$.

\medskip

For a hypergraph $H$, let $$\nu^*(H)=\max\left\{\sum_{e\in E(H)} w(e): w \mbox{
  is a fractional matching in $H$}\right\}.$$
A {\it fractional vertex cover} of $H$  is a
function $w:V(H)\rightarrow [0,1]$ such that for each $e\in E$,
$\sum_{v\in e}w(v)\ge 1$ . Let
$$\tau^*(H)=\min\left\{\sum_{v\in V(H)}w(v): w \mbox{ is a
  fractional vertex cover of $H$}\right\}.$$
Then the strong duality theorem of linear programming gives
$$\nu^*(H)=\tau^*(H).$$

We conclude this section by  proving the existence of a perfect fractional
matching in a uniform hypergraph whose independence number is not too large
(as otherwise such hypergraphs would be close to $H_k^{k-l}(U,W)$).

\begin{lemma}\label{4.4}
	Let $k,l$ be integers such that $k\ge 3$ and $k/2\le l <k$,
        and let  $\varepsilon, \xi$ be positive reals such that
        $\xi< (\varepsilon/5)^2 (3k)^{-4(k-l)}$.
	Let $n$ be a positive integer such that $n$ is sufficiently large and
	$n\in k \Z$.
     Let $H$ be a $k$-graph of order $n$ such that $\delta_{l}(H)> {n-l\choose k-l}-{n-l-n/k\choose k-l}-\xi n^{k-l}$
and $\alpha(H)\leq (1-1/k-\varepsilon/5)n$. Then
           $H$ contains a  perfect fractional matching.
\end{lemma}

\pf
	For convenience, let $V(H)=[n]$.
	Let $\omega$ be a minimum fractional vertex cover of $H$
	and we may assume that $\omega(1)\geq \omega(2)\ge \ldots\geq \omega(n)$.
        Let $E'=\{e\in {[n]\choose k}\ :\ e\notin E(H)\ \mbox{and }\sum_{i\in e}\omega(i)\geq 1\}$ and
	let $H'$ be obtained from $H$ by adding the edges in $E'$. Then $H'$ is stable and $\tau^*(H')=\tau^*(H)$.
	Thus  $\nu^*(H)=\nu^*(H')\geq \nu(H')$,
	and it suffices to show that $\nu(H')=n/k$,
	i.e., $H'$ contains a perfect matching.

	Let $S=[n]\setminus [n-l]$,
	and let $G$ be the hypergraph with $V(G)=[n]$ and $E(G)=N_{H'}(S)$, which is a
        $(k-l)$-graph on $[n]$. Since $H'$ is stable, $G$
        is also stable.        We may assume that
	
        \begin{itemize}
        \item [(1)] 	 $\nu(G)\le n/k-1$.
        \end{itemize}	
        For, otherwise, let $f_1, \ldots, f_{n/k}$ be a matching in $G$. Now $[n]\setminus \left(\bigcup_{i=1}^{n/k}f_i\right)$ is a set of size $(n/k)l$ and, hence, can
        be partitioned into $l$-sets, say $S_1, \ldots, S_{n/k}$. Since
        $H'$ is stable and $S\cup f_i\in E(H')$ for $i\in [n/k]$,  we
        have $S_i\cup f_i\in E(H')$ for $i\in [n/k]$. Hence,
        $\{S_i\cup f_i\ : \ i\in [n/k]\}$ is a perfect matching in $H'$. $\Box$
  \medskip

       We may also assume that
       \begin{itemize}
       \item [(2)]  $l\le k-2$.
       \end{itemize}
        For, suppose $l=k-1$. Then $G$ is a $1$-graph. Since $H'$ is
        stable and  $e(G)\ge \delta_{k-1}(H)\ge n/k-\lceil \xi n\rceil$, the first
      $n/k-\lceil \xi n\rceil$ vertices of $G$ are edges of
      $G$.

    Note that $H'- [n/k-\lceil \xi n\rceil]$ has
      $n-n/k+\lceil \xi n\rceil$ vertices. Since  $\alpha(H)\leq
      (1-1/k-\varepsilon/5)n$, $H'- [n/k-\lceil \xi n\rceil]$ has an
      edge. In fact, since $\xi< (\varepsilon/5)^2 (3k)^{-4(k-l)}$, we can greedily find
      pairwise disjoint edges
      $f_1,\ldots,f_{\lceil \xi n\rceil}$ in $H'- [n/k-\lceil \xi n\rceil]$.
     Since  $$n -(n/k-\lceil \xi n\rceil)- \lceil \xi n\rceil k=(k-1)(n/k- \lceil \xi n\rceil),$$
     we can partition $[n]\setminus  [n/k-\lceil \xi n\rceil] \setminus
\bigcup_{i=1}^{\lceil \xi n\rceil}f_i$ into $(k-1)$-sets $S_1,\ldots,
S_{n/k-\lceil\xi n\rceil}$. Now $S_i\cup \{i\}$, $i\in \left[
n/k-\lceil \xi n\rceil\right]$, form a matching in $H'$.  These edges and $\{f_1,\ldots,f_{\lceil \xi n\rceil}\}$ form a perfect matching in $H'$. $\Box$
	
\medskip

	Let $\eta=\varepsilon/(5k)$ and let $t= n/k -  \lfloor {\eta} n
        \rfloor$. For $i\in [n]$, we use $d_G(i)$ to denote the degree of $i$ in $G$.  We claim that

         \begin{itemize}
            \item [(3)] $d_G(t)>        \binom{n-1}{k-l-1}-\binom{n/(2k)}{k-l-1}$.
         \end{itemize}
        For suppose $d_G(t)\le     \binom{n-1}{k-l-1}-\binom{n/(2k)}{k-l-1}$.
  	Since $H'$ is stable,  $d_G(i)\le
        \binom{n-1}{k-l-1}-\binom{n/(2k)}{k-l-1}$ for $t\le i\le
        n/k$. Note that the degree of $t$ in $H_{k-l}^{k-l}([n]\setminus [n/k], [n/k])$ is ${n-1\choose k-l-1}$.
Thus,
   	\begin{eqnarray*}
		& & \left|E\left(H_{k-l}^{k-l}\left([n]\setminus [n/k],[n/k]\right)\right)\setminus E(G)\right|\\
		& \ge & \frac{1}{k-l}\left(\sum_{i=t}^{n/k} \left(d_{H_{k-l}^{k-l}([n]\setminus [n/k],[n/k])}(i)- d_G(i)\right) \right) \\
                &\ge & \frac{1}{k-l}	(n/k-t+1)	 \binom{n/(2k)}{k-l-1}\\
                & > & \frac{1}{k-l} \eta n (3k)^{-(k-l-1)}\binom{n}{k-l-1} \\
		& > & \sqrt \xi n^{k-l},
	\end{eqnarray*}
	as $\xi< (\varepsilon/5)^2 (3k)^{-4(k-l)}$.

	Hence $G$ is not $\sqrt{\xi}$-close to $H_{k-l}^{k-l}([n]\setminus [n/k], [n/k])$.
	However, since $G$ is stable and $n/k\le n/(2(k-l))$ (as $l\ge
        k/2$),  we may apply
        Lemma~\ref{stafrankl}
        with $n/k, k-l, \xi$ as
        $m,k,\xi$, respectively.
	So $\nu(G)\ge n/k$, contradicting (1). $\Box$

\medskip
	
        Note that $H'-[t]$ has $n-n/k+\lfloor {\eta} n \rfloor$
        vertices. Since  $\alpha(H)\leq (1-1/k-\varepsilon/5)n$,
        $H'-[t]$ has an edge. In fact, since  $\varepsilon n = 5k \eta
        n$,  $H'-[t]$ has  $\lfloor  {\eta} n \rfloor$ pairwise
        disjoint edges, say $f_1, \ldots f_{\lfloor \eta
          n\rfloor}$. Let $T=\bigcup_{i=1}^{\lfloor \eta n\rfloor} f_i$.

         Next we  find disjoint edges $e_1, \ldots,
       e_t$ of $G$  such that $|e_i\cap [t]|=1$ and $e_i\cap
       T=\emptyset$ for all $i\in
       [t]$.
       Suppose for some $s\in
       [t-1]$  we have found pairwise disjoint edges $e_1,\ldots, e_s$ of $G$ such that, for $i\in [s]$,  $e_i\cap [t]=\{i\}$ and $e_i\cap
       T=\emptyset$. The number of edges of $G$
       containing $s+1$ and intersecting $T\cup ([t]\setminus
       \{s+1\})\cup (\bigcup_{i=1}^se_i)$
       is at most $\binom{n-1}{k-l-1}-\binom{n-|T|-t-(k-l)s}{k-l-1}$. Note that
       $n-|T|-t-(k-l)s\ge n/(2k)$, as $l\ge k/2$. Hence, by (3),  there exists $e_{s+1}\in E(G)$ such
       that $e_{s+1}\cap [t]=\{s+1\}$, $e_{s+1}\cap T=\emptyset$, and
       $e_{s+1}$ is disjoint from $\bigcup_{i=1}^se_i$.

       Since $t=n/k-\lfloor \eta n\rfloor$,  $H-[t]-T-\bigcup_{i=1}^te_i$ has exactly $tl$ vertices (as $|e_i\cap [t]|=1$ for $i\in [t]$).
Partition the vertices in $H-[t]-T-\bigcup_{i=1}^te_i$ to pairwise disjoint $l$-sets
        $S_1,\ldots, S_t$. Then, since $H$ is stable, $S_i\cup e_i\in
        E(H')$ for $i\in [t]$. Hence, $\{f_i : i\in [\lfloor \eta
          n\rfloor]\}\cup \{S_j\cup e_j : j\in [t]\}$ is a perfect matching in $H'$.  \qed

\medskip

{\it Remark}. When we apply Lemma~\ref{stafrankl} in the end of the
proof of (3), we require $l\ge k/2$ so that $n/k\le n(2(k-l))$ (which
amounts to $m\le n/(2k)$ in Lemma~\ref{stafrankl}). This is not
necessary when $k=3$ and $l=1$, as we can use Lemma~\ref{3-graph-frac} (see Section 6)
which is the same as  Lemma~\ref{stafrankl} except with $m\le n/(2k)=n/4$ replaced by $m\le
n/2-1$.

\section{Almost perfect matchings}
To complete the proof of Theorem~\ref{main-thm}, we need to consider $n$-vertex $k$-graphs $H$
that are not close to $H_k^{k-l}(U, W)$ for any partition $U,W$ of
$V(H)$  with $|W|=m$ and $|U|=n-m$. This will be done using the
approach in Alon {\it et al.} \cite{AFHRRS12} to
find random subgraphs of $H$ and use them to find an almost perfect
matching in $H$.

We first use the absorbing lemma in Section 3 to find a
small matching $M_a$ in $H$ such that for any small set $S\subseteq V(H)$, $H[V(M_a)\cup S]$ has a nearly perfect matching. We then find
an almost perfect matching in $H-V(M_a)$ (see Lemma~\ref{almostperfect}), and use $M_a$ to absorb the unmatched vertices. To find this almost perfect  matching in $H-V(M_a)$, we  will find an
almost regular subgraph of $H$ with bounded  maximum 2-degree, so that the following result of
Frankl and R\"{o}dl \cite{FR85} can be applied. For any positive integer $l$, we use $\Delta_l(H)$ to denote the maximum $l$-degree of $H$.

\begin{lemma}[Frankl and R\"odl, 1985]\label{Rodl}
	For every integer $k \geq 2$ and any real $\varepsilon>0$,
	there exist $\tau=\tau(k,\varepsilon)$
	and $d_0=d_0(k,\varepsilon)$ such that,
	for every $n \ge D \ge d_0$ the following holds:
	Every $k$-graph on $n$ vertices with $(1-\tau)D <d_H(v)<(1+\tau)D$ and $\Delta_2(H) <\tau D$
	contains a matching covering all but at most $\varepsilon n$ vertices.
\end{lemma}

In order to find a subgraph in a $k$-graph satisfying conditions in Lemma~\ref{Rodl}, we use the same two-round randomization
technique as in \cite{AFHRRS12}. The only difference is that in the first
round, we also need  to  bound the independence number of the subgraph
(in order to deal with  hypergraphs not close to $H_k^{k-l}(U,W)$), which we were not able to do
using the second  moment method alone. Here we use the hypergraph
container result of Balogh {\it et al.} \cite{BMS15}. (A similar result is proved independently
by Saxton and Thomason \cite {ST15}.)
To state that result, we need additional terminology.

A family ${\cal F}$ of subsets    of a set $V$  is
said to be {\it increasing} if, for any $A\in \F$ and
$B\subseteq V$,   $A\subseteq B$ implies $B\in \F$.
Let $H$ be a hypergraph. We use $v(H), e(H)$ to denote the number of
vertices,  number of edges in $H$, respectively, and use $\I(H)$ to
denote the collection of all independent sets in $H$. Let $\varepsilon>0$,
and  ${\cal F}$ be a family of subsets of $V(H)$. We say that $H$ is \textit{$(\mathcal F , \varepsilon)$-dense} if $e(H[A])\ge \varepsilon e(H)$ for every $A\in \F$. We use $\overline{\F}$ to denote the family consisting of subsets of $V(H)$ not in $\F$.

\begin{lemma}[Balogh, Morris, and Samotij, 2015]\label{thm2.2}
        For every $k \in \N$ and all positive $c$ and $\varepsilon$, there exists a positive constant $C$ such that the following holds.
        Let $H$ be a $k$-graph and let $\F$ be an increasing family of subsets of $V(H)$ such that
        $|A| \ge \varepsilon v(H)$ for all $A \in \F$.
        Suppose that $H$ is $(\mathcal F , \varepsilon)$-dense and $p \in (0, 1)$ is such that, for every $l \in [k]$,
        \[\Delta_l(H)\le cp^{l-1}\frac{e(H)}{v(H)}.
        \]
        Then there exist a family $\S\subseteq \binom{V(H)}{\le Cp v(H)}$ and functions $f: \S \to \overline{\F}$ and $g: \I (H) \to \S$ such that, for every $I\in \I (H)$,
        \[g(I)\subseteq I \quad and \quad I\setminus g(I)\subseteq f(g(I)).\]
\end{lemma}

Before we apply Lemma~\ref{thm2.2} to control the independence
number of a random subgraph of certain $k$-graphs, we show for any
$\varepsilon >0$ there exist $\rho', \varepsilon' >0$ such that if an $n$-vertex
$k$-graph $H$ is not $\varepsilon$-close to $H_k^{k-l}(U,W)$ for any partition $U,W$ of $V(H)$ with $|U|=n-m$
and $|W|=m$,
and if $\delta_{l}(H)\geq {n-l\choose k-l}-{n-l-m\choose k-l}-\rho' n^{k-l}$ then  $H$ should be
$(\F,\varepsilon')$-dense for some  family $\F$ consisting of large
subsets of $V(H)$.

\begin{lemma}\label{4.1}
	Let $k,l$ be integers such that $k\ge 2$ and $l\in [k-1]$. Let
        $0< \varepsilon\ll 1$, $\rho'\le \varepsilon/8$, and
        $0<\mu\le \varepsilon/40$. Let $m,n$ be
        sufficiently large integers such that $ n/k-\mu n\le m\le n/k$. Suppose $H$ is a $k$-graph with order $n$ such that
        	$\delta_l(H)> {n-l\choose k-l}-{n-l-m\choose k-l}-\rho' n^{k-l}$, and $H$ is not
        $\varepsilon$-close to $H_k^{k-l}(U,W)$ for any partition
        $U,W$ of $V(H)$ with $|W|=m$. Then $H$ is
        $(\F, \varepsilon/(2k!))$-dense, where $\F=\{A\subseteq V(H) : |A|\ge (1-1/k-\varepsilon/4)n\}$.
\end{lemma}

\pf   Suppose to the contrary that there exists $A\subseteq V(H)$ such
       that $|A|\ge (1-1/k-\varepsilon/4)n$ and $e(H[A])\le \varepsilon e(H)/(2k!)$. By removing vertices if necessary,
        we may choose such $A$  that  $|V(H)\setminus A|\ge m$ (as $m\le n/k$), and let $W\subseteq V(H)\setminus A$ such that
        $|W|=m$.  For convenience, let  $B=V(H)\setminus W\setminus A$. Then
        \[
          	|B|\le n - m- (1-1/k-\varepsilon/4)n \le \varepsilon n/4 + n/k-(1/k-\mu)n\le
        11\varepsilon n/40.
        \]
        Let $U=V(H)\setminus W$ and
        $H_0=H_{k}^{k-l}(U,W)$. We derive a contradiction by showing
        that $|E(H_0)\setminus E(H)|<\varepsilon n^k$.

        Note that, for each $f\in E(H_0)\setminus E(H)$, we have $1\le |f\cap
        W|\le k-l$ (by definition of $H_0$);  so  $|f\cap B|> 0$ or $|f\cap A|\ge l$. Thus
          \[
             |E(H_0)\setminus      E(H)|\le |\{f\in E(H_0): |f\cap B|>0\}|+|\{f\in E(H_0)\setminus E(H): |f\cap A|\ge l\}|.
        \]
        It is easy to see that
        \[
             |\{f\in E(H_0): |f\cap B|>0\}|\le |B||W|n^{k-2}\le (11\varepsilon n/40)(n/k)n^{k-2}=\frac{11\varepsilon}{40k} n^k.
        \]

	Next, we bound $|\{f\in E(H_0)\setminus E(H): |f\cap A|\ge l\}|$.  Fix an
        arbitrary  $l$-set $S\subseteq A$.     Note that
     \[
        |\{f\in E(H): S\subseteq f\mbox{ and } f\cap B\ne \emptyset\}|\le |B|
        n^{k-l-1}\le \frac{11\varepsilon}{40} n^{k-l}.
    \]
      For  any  $f\in E(H)$ and $S\subseteq f$, we have  $f\cap B\ne
        \emptyset$, or $f\subseteq A$, or $f\in E(H_0)$. So
          \begin{eqnarray*}
             & & |\{f\in E(H) : S\subseteq f \mbox{ and } f\in E(H_0)|\\
             &\ge & d_H(S)-|\{f\in E(H) : S\subseteq f  \mbox{ and } f\cap B\ne \emptyset\}| -|\{f\in E(H) : S\subseteq f  \mbox{ and } f\subseteq A\}|\\
             & \ge & d_H(S) -\frac{11\varepsilon}{40} n^{k-l}-d_{H[A]}(S).
          \end{eqnarray*}

Hence,
\begin{eqnarray*}
& & |\{f\in E(H_0)\setminus E(H) : |f\cap A|\ge l\}|\\
&\le & \sum_{S\in {A\choose l}}|\{f\in E(H_0)\setminus E(H): S\subseteq f\}|\\
& \le &\sum_{S\in {A\choose l}}\left( d_{H_0}(S)- |\{f\in E(H): f\in E(H_0) \mbox{ and } S\subseteq f\}|\right) \\
&\le & \sum_{S\in {A\choose l}}\left(d_{H_0}(S)- d_H(S)+ \frac{11\varepsilon}{40} n^{k-l}+ d_{H[A]}(S)\right).
\end{eqnarray*}
Note that for $S\in {A\choose l}$, $d_{H_0}(S)={n-l\choose
  k-l}-{n-l-m\choose k-l}$ and, hence,
$d_{H_0}(S)-d_H(S)<\rho' n^{k-l}$ by the assumption on $\delta_l(H)$. Hence,
        \begin{align*}
		|E(H_0)\setminus E(H)|
		&< \frac{11\varepsilon}{40k} n^{k} + \binom{|A|}{l}\left(\rho' + \frac{11\varepsilon}{40}\right)n^{k-l}+\sum_{S\in \binom{A}{l}} d_{H[A]}(S)\\
		&\le \left(\frac{11\varepsilon}{40k} + \rho' + \frac{11\varepsilon}{40}\right) n^k + \binom{k}{l}e(H[A])\\
		&\le  \left(\frac{11}{120} + \frac{1}{8} + \frac{11}{40}\right) \varepsilon n^k + \binom{k}{l}\frac{\varepsilon n^k}{2 k!}\quad \mbox{(since $k\ge 3$ and $\rho'\le\varepsilon/8$)}\\
		&< \varepsilon n^k,
	\end{align*}
a contradiction. \qed

\medskip

We now use Lemma~\ref{thm2.2} to show that one
can control, with high probability, the independence number of
a subgraph of a $k$-graph induced by a random subset of vertices.

\begin{lemma}\label{indep}
        Let $c, \varepsilon', \alpha$ be positive reals and let $k,n$ be positive integers. Let $H$ be an $n$-vertex $k$-graph such that  $e(H)\ge cn^k$  and $e(H[S])\ge
        \varepsilon' e(H)$ for all $S\subseteq
        V(H)$ with $|S|\ge \alpha n$.
         Let $R\subseteq V(H)$ be obtained  by taking each vertex of
           $H$ uniformly at random with probability $n^{-0.9}$.
        Then for any positive real $\gamma \ll \alpha$, the independence number $\alpha(H[R])$ of $H[R]$ satisifes $\alpha(H[R])\le (\alpha +\gamma+o(1))n^{0.1}$
with probability at least $1-n^{O(1)}e^{-\Omega (n^{0.1})}$.
\end{lemma}

   \pf
        Define $\F :=\left\{A\subseteq V(H) \ : \  e(H[A])\ge \varepsilon'
          e(H) \mbox{ and } |A|\ge \varepsilon' n\right\}$.
        Then $\F$ is an increasing family,  and $H$ is $(\F, \varepsilon')$-dense.
        Let $p=n^{-1}$ and $v(H)=n$. Then
        \[\Delta_l(H)\le \binom{n}{k-l}\le n^{k-l}\le c^{-1}n^{-l}e(H)=c^{-1} p^{l-1}\frac{e(H)}{v(H)}.
        \]
        Thus by Lemma~\ref{thm2.2}, there exist a constant $C$
        (depending only on $\varepsilon'$ and $c$),  a family $\S\subseteq \binom{V(H)}{\le C}$, a function $f: \S\to \overline{\F}$, and
        a family $\T:= \left\{F\cup S :  F\in f(\S) , S\in \S \right\}$,
        such that every independent set in $H$ is contained in some $T\in \T$.
        Since $\S\subseteq \binom{V(H)}{\le C}$, $|\S|\le C n^C$
        and, hence, $$|\T|=|\S| |f(\S)|\le |\S|^2\le C^2n^{2C}.$$

        We claim that $|T|< \alpha n+C$ for all $T\in \T$. To see this, let $T=F\cup S$ for some $F\in f(\S)$ and $S\in \S$. By definition, $F\in \overline{\F}$
        and, hence, $e(H[F])<\varepsilon' e(H)$. Since $e(H[S])\ge
        \varepsilon' e(H)$ for any $S\subseteq V(H)$ with $|S|\ge \alpha n$,
        we have $|F|<\alpha n$.
        Therefore,
           $|T|\le |F|+|S|<\alpha n +C.$
	
        We wish to apply Lemma~\ref{chernoff} and, hence, we need to make sets in
        $\T$ slightly larger.
        Take an arbitrary map $h : \T \to {V(H)\choose \lfloor \alpha n+C\rfloor}$ such
        that  $T\subseteq h(T)$ for all $T\in \T$, and let $\T'= h(\T)$.
        Then $$|\T'|\le |\T|\le |\S|^2\le C^2n^{2C}.$$
        Note that for each fixed $T'\in \T'$, we have $|R\cap T'|\sim Bi\left(|T'|, n^{-0.9} \right)$ and
	$\E (|R\cap T'|)=n^{-0.9}|T'| =\lfloor \alpha n +C\rfloor n^{-0.9}.$
        We apply Lemma~\ref{chernoff} to $|R\cap T'|$ by taking
        $\lambda= \gamma n^{0.1}$, where $\gamma$ is given and  $\gamma\ll\alpha$.
        Then \[\P\left(\big| |R\cap T'| - n^{-0.9}|T'|  \big| \ge \lambda  \right) \le e^{-\Omega(\lambda^2/ (n^{-0.9}|T'|)}= e^{-\Omega(n^{0.1})}.\]
        So with probability at most $e^{-\Omega(n^{0.1})}$, we have $|R\cap T'|\ge n^{-0.9}|T'|+\lambda$;
       hence, $|R\cap T'|\ge (\alpha +\gamma + C/n)n^{0.1}$ with probability at most $e^{-\Omega(n^{0.1})}$.

        Therefore,
        with
        probability at most $C^2n^{2C}e^{-\Omega(n^{0.1})}$ (from union bound), there exists some $T'\in \T'$ such that $|R\cap T'|\ge (\alpha +\gamma +C/n)n^{0.1}$.
        Hence, with probability at least $1-C^2n^{2C}e^{-\Omega(n^{0.1})}$,
        $|R\cap T'|<(\alpha+\gamma  +C/n)n^{0.1}$ for all $T'\in \T'$.

	 It remains to show that, conditioning on that $|R\cap T'|<(\alpha +\gamma+ C/n)n^{0.1}$ for all $T'\in \T'$,
        $|J|\le (\alpha +\gamma+C/n)n^{0.1}$ for every independent set $J$ in $H[R]$.
        Since such $J$ is also an independent set in $H$, there
        exist  $T\in \T$ and $T'\in \T'$ such that $J\subseteq T\subseteq T'$.
        Thus $J\subseteq R\cap T'$ and $|J|\le |R\cap T'|< (\alpha +\gamma+C/n)n^{0.1}$.

        Hence, $\alpha(H[R])\le (\alpha +\gamma+C/n)n^{0.1}$, with probability at least $1-C^2n^{2C} e^{-\Omega (n^{0.1})}$. \qed

\medskip

    The following result is the outcome of the first round of the two-round randomization procedure in \cite{AFHRRS12}.
    We summarize this round as a lemma (see the proof of
    Claim 4.1 in \cite{AFHRRS12}) and outline a proof, since we need to make some small modifications.  Here we adopt the notation in
    \cite{AFHRRS12}.

    \begin{lemma}\label{lem1-5}
        Let $k>d>0$ be integers with $k\ge 3$
       and let $H$ be a $k$-graph on $n$ vertices.
	Take $n^{1.1}$ independent copies of $R$ and denote them by $R^i$, $1\le i\le n^{1.1}$, where $R$ is chosen from $V(H)$ by taking each vertex uniformly at random with probability $n^{-0.9}$ and then deleting less than $k$ vertices  uniformly at random so that $|R|\in k\Z$.
        For each $S\subseteq V(H)$, let $Y_S:=|\{i: \ S\subseteq R^i\}|$ and $\bd^i_S:=|N_H(S)\cap R^i|$.
            Then with probability at least $1-o(1)$,  all of the following statements hold:
        \begin{itemize}
              \setlength{\itemsep}{0pt}
		\setlength{\parsep}{0pt}
		\setlength{\parskip}{0pt}
            \item [$(i)$] for every $v\in V(H)$, $Y_{\{v\}}=(1+o(1)) n^{0.2}$ 
            \item [$(ii)$] $Y_{\{u,v\}}\le 2$ for every pair $\{u, v\} \subseteq V(H)$,
            \item [$(iii)$] $Y_e\le 1$ for every edge $e \in E(H)$,
            \item [$(iv)$] for all $i= 1, \dots ,n ^{1.1}$, we have
              $|R^i| =(1+o(1))n^{0.1}$, and 
            \item [$(v)$] 
            if  $\mu, \rho'$ are constants with $0<\mu\ll \rho'$, $n/k-\mu n\le m\le n/k$, and
             $\delta_d(H)\ge {n-d\choose
                   k-d}-{n-d-m\choose k-d}-\rho' n^{k-l}$,  then
                for all $i= 1, \dots ,n ^{1.1}$ and all $D\in { V(H) \choose d}$ and for any positive real $\xi\ge 2\rho'$, we have
               $$\bd^{i}_D >  {|R^i|-d\choose k-d}-{|R^i|-d-|R^i|/k \choose k-d}-\xi |R^i|^{k-d}.$$
        \end{itemize}
    \end{lemma}

\begin{proof} 	
	Note that the removal of less than $k$ vertices does not affect $(i)$ to $(iv)$.
    Also note that  $|Y_S| \sim Bi(n^{1.1},n^{-0.9|S|})$ for $S\subseteq V(H)$.

    Thus,  $\E(|Y_{\{v\}}|)=n^{0.2}$ for $v\in V(H)$, and it follows from Lemma~\ref{chernoff} that
      \[
      		\P\left(\big|Y_{\{v\}}- n^{0.2}\big|>n^{0.15} \right)\le e^{-\Omega(n^{0.1})}
	\] 	
	  Hence $(i)$ holds with probability at least $1- e^{-\Omega(n^{0.1})}$.
	
	  To prove $(ii)$, let
	  \[Z_2=\bigg| \left\{ \{u,v\}\in \binom{V(H)}{2}: Y_{\{u,v\}} \ge 3\right\}\bigg|, \]
	  and for $k\ge 3$, let
	  \[ \quad Z_k=\bigg| \left\{ S\in \binom{V(H)}{k}: Y_{S} \ge 2\right\}\bigg|.
	  \]
	  Then $\E(Z_2)<n^2(n^{1.1})^3(n^{-0.9})^6=n^{-0.1}$ and $\E(Z_k)<n^k(n^{1.1})^2(n^{-0.9})^{2k}=n^{2.2-0.8k}\le n^{-0.2}$ (for $k\ge 3$).
      By Markov's inequality,
	  \[\P(Z_2=0)>1-n^{-0.1} \mbox{ and, for $k\ge 3$, } \P(Z_k=0)>1-n^{-0.2}.\]
	  Thus $(ii)$ and $(iii)$ hold with probability at least $1-n^{-0.1}$ and $1-n^{-0.2}$, respectively.

         By Lemma~\ref{chernoff} (with $\lambda=n^{0.095}$), we have
         \[ \P \left(\big| |R^i|- n^{0.1}  \big| \ge n^{0.095} \right)\le e^{-\Omega(n^{0.09})}
         \]
         for each $i$. Thus by union bound, $(iv)$ holds with probability at least $1-n^{1.1} e^{-\Omega(n^{0.09})}$.

         Next, we prove $(v)$.
         Conditioning on $\big| |R^i|- n^{0.1}  \big| < n^{0.095}$ for all $i$ and using the assumption that ,
          $0<\mu \ll \rho'$, $n/k-\mu n \le m\le n/k$,
         and $n$ is large,
         we have
         \begin{align*}
                  &\ \left(\binom{n-d}{k-d}-{n-d-m \choose k-d}-\rho' n^{k-d}\right) (n^{-0.9})^{k-d} \\
          	\ge  &\ \binom{|R^i|-d}{k-d}-{|R^i|-d-|R^i|/k\choose k-d}-1.5\rho'|R^i|^{k-d}.
          \end{align*}
          So for each $D\in {V(H)\choose d}$
          and each fixed $R^i$,
          \begin{align*}
           \E(\bd^i_D) & = (1-o(1))d_H(D) (n^{-0.9})^{k-d}\\
                              &\ge (1-o(1)) \left(\binom{n-d}{k-d}-{n-d-m \choose k-d}-\rho' n^{k-d}\right) (n^{-0.9})^{k-d} \\
                             &\ge  (1-o(1)) \left(\binom{|R^i|-d}{k-d}-{|R^i|-d-|R^i|/k\choose k-d}-1.5\rho'|R^i|^{k-d}\right) \\
				& \ge \binom{|R^i|-d}{k-d}-{|R^i|-d-|R^i|/k\choose k-d}-1.8\rho'|R^i|^{k-d}.
          \end{align*}
                  In particular
         \[ \E(\bd^i_D)= \Omega(n^{0.1(k-d)}).\]

         We apply Janson's Inequality (Theorem 8.7.2 in \cite{AS08}) to bound the deviation of $\bd^i_D$.
         Write $\bd^i_D=\sum_{e\in N_H(D)} X_e$,
         where $X_e=1$ if $e\subseteq R^i$ and $X_e=0$ otherwise. Then
         \[\Delta = \sum_{e\cap f \ne \emptyset} \P(X_{e}=X_{f}=1)\le \sum_{l=1}^{k-d-1} p^{2(k-d)-l}\binom{n-d}{k-d}\binom{k-d}{l}\binom{n-k}{k-d-l}
         \]
         and, thus, $ \Delta = O(n^{0.1(2(k-d)-1)})$.
         By Janson's inequality, for any $\gamma>0$,
         \[\P(\bd^i_D\le(1-\gamma)\E(\bd^i_D))\le e^{-\gamma^2 \E(DEG^i_D)/(2+\Delta/\E(DEG^i_D))}=e^{-\Omega(n^{0.1})}.
         \]
         Since $\xi\ge 2\rho'$, by taking $\gamma$ small, the  union bound shows that, with probability
         at least $1-n^{d+1.1}e^{-\Omega(n^{0.1})}$,
         $$DEG_D^i\ge \binom{|R^i|-d}{k-d}-{|R^i|-d-|R^i|/k\choose k-d}-\xi|R^i|^{k-d}.$$

         Thus, it follows from (4) that $(v)$ holds with probability at least $$(1-n^{1.1}e^{-\Omega(n^{0.1})})
         (1- n^{d+1.1}e^{-\Omega(n^{0.1})})>1- (n^{1.1}+n^{d+1.1})e^{-\Omega(n^{0.1})}.$$

         Hence, it follows from  union bound that,  with probability at least
         \[1-e^{-\Omega(n^{0.1})}-n^{-0.1}-n^{-0.2}-n^{1.1}e^{-\Omega(n^{0.09})}-(n^{1.1}+n^{d+1.1})e^{-\Omega(n^{0.1})}=1-o(1),
          \]
 $(i)$-$(v)$ hold.
\end{proof}

   We summarize the second round randomization in  \cite{AFHRRS12}  as the following lemma (again, see the proof of Claim 4.1 in \cite{AFHRRS12}).

    \begin{lemma}\label{2-degree}
	   Assume $R^i$, $i=1,\ldots, n^{1.1}$, satisfy $(i)$-$(v)$ in
           Lemma~\ref{lem1-5}, and that each $R^i$ has a  perfect fractional matching $w^i$.
          Then there exists a spanning subgraph $H''$ of $H$ such that $\d_{H''}(v)=(1+o(1))n^{0.2}$ for each $v\in V$, and $\Delta_2(H'')\le n^{0.1}$.
    \end{lemma}

\medskip

      We are now ready to show that for an $n$-vertex  $k$-graph $H$ satisfying the
      conditions of Theorem~\ref{main-thm} and not $\varepsilon$-close to $H_k^{k-l}(U,W)$ for any partition $U,W$ of $V(H)$ with $|W|=m$, after taking away an
      absorbing matching $M_a$ (from Lemma~\ref{Absorb-lem}), the
      resulting $k$-graph has an almost perfect matching.

\begin{lemma}\label{almostperfect}
      Let $k,l$ be integers such that $k\ge 3$ and $k/2\le l<k$.
            Let
      $\rho', \varepsilon, \sigma, \mu$ be positive reals such that $\rho'< \varepsilon^2(3k)^{-4(k-l)}/100$ and $\mu \le \varepsilon /40$.
      Let $n,m$ be sufficiently large integers such that $n/k-\mu n\le m\le n/k$.
      Suppose  $H$ is a $k$-graph on $n$ vertices such that $\delta_l(H)\ge
      {n-l\choose k-l}-{n-l-m\choose k-l} -\rho' n^{k-l}$, and $H$ is not $\varepsilon$-close to
       $H_k^{k-l}(U,W)$ for any partition $U,W$ of $V(H)$ with $|W|=m$. Then
$H$ contains a matching covering all but at most $\sigma n$ vertices.
\end{lemma}

\begin{proof}
   By Lemma~\ref{4.1}, $e(H[S])\ge (\varepsilon/(2k!)) e(H)$
    for all $S\subseteq V(H)$ with $|S|\ge \alpha n$, where $\alpha=1-1/k-\varepsilon/4$.
    Note that $$e(H)=\delta_0(H)\ge
    \binom{n}{l}\delta_l(H)/\binom{k}{l}\ge cn^k,$$
    where $c>0$ is a constant and $c\ll 1/{k\choose l}$.

    Let $R, R^i$ be given as in Lemma~\ref{lem1-5}.
    Then it follows from  Lemma~\ref{indep} that, with  probability
    $1-o(1)$,    $H[R^i]$ has independence number $\alpha(H[R^i])\le
    (\alpha +o(1)+\gamma)n^{0.1}$ for all $i$,
    where $\gamma \ll \alpha$.
    Additionally by $(v)$ of Lemma~\ref{lem1-5},
    $\delta_l(H[R^i]) > {|R^i|-d\choose k-d}-{|R^i|-d-|R^i|/k \choose k-d}-\xi |R^i|^{(k-d)}$ for any $\xi \ge 2\rho'$.
     Thus by Lemma~\ref{4.4}, with probability $1-o(1)$, for each $i$, $H[R^i]$ has a perfect fractional matching.

 Hence by Lemma~\ref{2-degree}, $H$ has a spanning subgraph $H''$ such that $\d_{H''}(v)=(1+o(1))n^{0.2}$ for each $v\in V$, and $\Delta_2(H'')\le n^{0.1}$.
Thus we may apply Lemma~\ref{Rodl} to find a matching covering all but
at most $\sigma n$ vertices in $H''$, for sufficiently large $n$.
\end{proof}

\section{Conclusion}

      In this section, we complete the proof of Theorem~\ref{main-thm} and discuss some related work.

\medskip

        {\it Proof of Theorem~\ref{main-thm}}.
        By Lemmas~\ref{small-matching} and \ref{Phk}, we may assume that for any $0<\varepsilon<(8^{k-1}k^{5(k-1)}k!)^{-3}$, $H$ is not $\varepsilon$-close to $H_k^{k-l}(U,W)$ for any
          partition $U,W$ of $V(H)$ with  $|W|= m$.

        By Lemma~\ref{Absorb-lem}, there exist constants $c'= c'(k,l)$ and $\rho= \rho(c',k,l,\varepsilon)$ small enough,
        such that for positive integers $a,h$ satisfying $h\le l$, $a\le k-l$, and $al\ge a(k-l)+(k-h)$,
       there exists a matching $M_a$ of size at most $2k\rho n$ with the following property:
	    For any subset $S\subseteq V(H)$ with $|S|\le c'\rho n$,
        $H[V(M_a)\cup S]$ has a matching covering all but at most $al+h-1$ vertices.

          Now consider $H_1=H-V(M_a)$. Then $\delta_{l}(H_1)\ge \delta_l(H)-(2k^2\rho n)n^{k-l-1}$. Let $\rho_1= 4k^2 \rho$ and
          $n_1=n-k|M_a|$.
          Then $$\delta_l(H_1)\ge {n_1-l\choose k-l}-{n_1-l-m \choose k-l}-\rho_1 n_1^{k-l}$$
           and  $H_1$ is not $(\varepsilon/2)$-close to $H_k^{k-l}(U,W)$ for any
          partition $U,W$ of $V(H_1)$ with $|W|= m$, since $n$ is large enough and $\rho\ll \varepsilon$.

          By Lemma~\ref{almostperfect}, $H_1$ has a matching $M_1$ such that $|V(H_1)\setminus V(M_1)|<c'\rho n_1\le c'\rho n$.
          Then there exists a matching $M_2$ in
          $H_2:=H[V(M_a)\cup (V(H_1)\setminus V(M_1))]$ such that $|V(H_2)\backslash V(M_2)|\le al+h-1$.

          Now $M_1\cup M_2$ is a matching in $H$ covering all but at most $al+h-1$ vertices of $H$.
       By taking $a=\lceil(k-l)/(2l-k)\rceil$ and $h=k-a(2l-k)$, which  minimizes $al+h-1$, we see that $M_1\cup M_2$ is the desired matching. \qed

\medskip

There are two places in the proof of Theorem~\ref{main-thm} where we require $l>k/2$: Lemma~\ref{Absorb-lem} for absorbing matching and Lemma~\ref{4.4} for perfect fractional
matchings. We do not know how to derive such results for $l\le
k/2$. However, for $k=3$ and $l=1$, the absorbing part can be taken care of by
the following  result of H\'an, Person, and Schacht
\cite{HPS09}.

\begin{lemma}[H\`an, Person, and Schacht]\label{Ab-HPS}
Given any $\gamma>0$, there exists an integer $n_0=n_0(\gamma)$ such that
the following holds. Suppose that $H$ is a 3-graph on $n\geq n_0$ vertices such that $\delta_1(H)\geq (1/2+2\gamma){n\choose 2}$. Then there is a matching $M$ in $H$ of size $|M|\leq \gamma^3n/3$ such that for every set $V'\subseteq V(H)-V(M^*)$
with $|V'|\leq \gamma^6n $,  there is a matching in $H$ covering precisely the vertices in  $V'\cup V(M^*)$.
\end{lemma}

For the perfect fractional matching part, we need a
result of Berge \cite{Be58} on maximum matchings.
 For a graph $G$, we use $c_o(G)$ to denote the number of odd components in $G$.

\begin{lemma}[Berge] \label{graphmatching}
Let $G$ be a graph on $n$ vertices. Then
$$\nu(G)=\min\left\{\left(n-c_o(G-W)+|W|\right)/2  :
W\subseteq V(G)\right\}.$$
\end{lemma}

\begin{lemma}\label{3-graph-frac}
Let $c, \rho$ be constant such that $0<\rho\ll 1$ and $0<c<1/2$, and let $m, n$ be positive integers such that
$n$ is sufficiently large and $cn\leq m\leq n/2-1$. Let $G$ be a 2-graph with $V(G)=[n]$ such that  $\nu(G)\leq m$ and
$G$ is stable with respect to the natural order on $[n]$.  If $e(G)> {n\choose
  2}-{n-m\choose 2}-\rho n^2$, then $G$ is $2\sqrt{\rho}$-close to $H_2^2([n]\setminus[m],[m])$.
\end{lemma}

\pf  Since $G$ is stable, we have
\begin{itemize}
\item [(1)] $N_G(i)\setminus \{j\}\subseteq N_G(j)\setminus \{i\}$ for any $i,j\in [n]$ with $i>j$.
\end{itemize}

 By Lemma \ref{graphmatching}, there exists $W\subseteq V(G)$ such that
$$\nu(G)=\left(n-c_o(G-W)+|W|\right)/2.$$
We choose maximal such $W$, and
let $C_1,\ldots, C_q$ denote the components of $G-W$. Without loss of
generality, assume $|V(C_1)|\geq \cdots\geq |V(C_q)|$, and let
$c_i:=|V(C_i)|$ for $i\in [q]$.
Then
\begin{itemize}
\item [(2)] $q=c_o(G-W)$, i.e.,  $c_i$ is odd for all $i\in [q]$.
\end{itemize}
For, otherwise, suppose that $c_i$ is even for some $i\in [q]$.
Let $x\in V(C_i)$ and $W':=W\cup \{x\}$. Then $c_o(G-W')\ge c_o(G-W)+1$. This forces
$\left(n-c_o(G-W)+|W|\right)/2 =\left(n-c_o(G-W')+|W'|\right)/2$, as
$\nu(G)=\left(n-c_o(G-W)+|W|\right)/2$. But then,
$W'$ contradicts the choice of $W$. $\Box$

\begin{itemize}
\item [(3)] $c_i=1$ for $i=2, \ldots, q$.
\end{itemize}
For, suppose   $c_2\geq 2$.  Then $c_1\geq c_2\geq 2$; so
there exist $a_1b_1\in E(C_1)$ and $a_2b_2\in E(C_2)$. If
$a_1> a_2$ then $a_1b_2\in E(G)$ by (1), and  if $a_1<a_2$ then $b_1a_2\in
E(G)$ by (1). So  there is edge between $C_1$ and $C_2$,
contradicting the fact that $C_1$ and $C_2$ are different components
of $G-W$. $\Box$

\medskip

By (3), we have
\begin{align*}
m\geq \nu(G)&=\left(n-(c_o(G-W)-|W|)\right)/2\\
&=\left((c_1+|W|+q-1)-(q-|W|)\right)/2\\
&=(c_1-1)/2+|W|.
\end{align*}
Thus $|W|\le m-(c_1-1)/2$. Hence,
$$e(G)  \leq  {n\choose 2}-{n-|W|\choose 2}+{c_1\choose 2}\leq {n\choose 2}-{n-m+(c_1-1)/2\choose 2}+{c_1\choose 2}.$$
Since $e(G)>{n\choose 2}-{n-m\choose 2}-\rho n^2$, we have
\begin{align*}
{n-m\choose 2}+\rho n^2 & >  {n-m+(c_1-1)/2\choose 2}-{c_1\choose 2}\\
&={n-m\choose 2}+\frac{1}{8}(c_1-1)^2+\frac{1}{4}(c_1-1)(2n-2m-1)-{c_1\choose 2},
\end{align*}
which gives
\begin{align*}
-\frac{3}{8}(c_1-1)^2+\frac{1}{4}(c_1-1)(2n-2m-3)< \rho n^2.
 \end{align*}
Hence, $c_1< \sqrt{\rho}n$, since $\rho \ll 1$ and $m\le n/2-1$

Note that every edge of $G$ intersects $W\cup V(C_1)$. So by (1), every edge of $G$ intersects $[|W|+c_1] \subseteq [m+(c_1+1)/2]\subseteq [m+\sqrt{\rho}n/2]$.
Since $e(G)> {n\choose
  2}-{n-m\choose 2}-\rho n^2$, we have
\begin{align*}
|E(H_2^2([n]\setminus[m],[m]))\backslash E(G)| & \leq 2\sqrt{\rho}n^2.
\end{align*}
This completes the proof of the lemma. \qed

\medskip

Thus, using Lemma~\ref{3-graph-frac} instead of Lemma~\ref{stafrankl} in the end of the
proof of (3) for Lemma~\ref{4.4}, we see that Lemma~\ref{4.4} holds in the case
when $k=3$ and $l=1$.
Thus, our approach (using Lemma~\ref{Ab-HPS}
instead of Lemma~\ref{Absorb-lem}) gives an
alternative proof of the following result of K\"{u}hn, Osthus, and
Townsend \cite{KOT13} (and independently by Khan\cite{Kh13}) on
perfect matchings in 3-graphs.

\begin{theorem} [K\"{u}hn, Osthus, and
Townsend; Khan] \label{thm-KOT}
There exists  $n_0\in \mathbb{N}$ such that if $H$ is a 3-graph of order $n\geq n_0$, $m\le n/3$,
and $\delta_1(H)>{{n-1}\choose 2}-{{n-m}\choose 2},$
then $\nu(H)\ge m$.
\end{theorem}

For the general case,  H\`{a}n, Person, and Schacht \cite{HPS09} and, independently, K\"{u}hn, Osthus, and Townsend \cite{KOT14} conjectured
that the asymptotic $l$-degree threshold for a perfect matching in a $k$-graph with $n$ vertices is
 $$\left(\max \left\{\frac 1 2, 1-\left(1-\frac 1 k\right)^{k-l}\right\}+o(1)\right){n-l\choose k-l}.$$
The first term $(1/2+o(1)){n-l\choose k-l}$ comes from a parity construction: Take disjoint nonempty sets
$A$ and $B$ with $||A|-|B||\le 2$, form a hypergraph $H$ by taking all $k$-subsets $f$ of
$A\cup B$ with $|f\cap A|\not\equiv |A| \pmod 2$.
The second term is given by the hypergraph obtained from $K_n^k$ (the complete $k$-graph on $n$ vertices) by deleting
all edges from a subgraph $K_{n-n/k+1}^k$.


\begin{thebibliography}{99}



\bibitem {AFHRRS12} N. Alon, P. Frankl, H. Huang, V. R\"{o}dl, A. Ruci\'{n}ski, and B. Sudakov,
 Large matchings in uniform hypergraphs and the conjectures of Erd\H{o}s
 and Samuels,
{\it J. Combin. Theory Ser. A}, {\bf 119} (2012), 1200--1215.

\vspace*{-1ex}

\bibitem{AS08}
N. Alon and J. Spencer, The Probabilistic Method, Wiley-Intersci. Ser. Discrete Math. Optim., John Wiley  Sons, Hoboken, NJ, 2000, third edition, 2008.

\vspace*{-1ex}
\bibitem{BMS15}
J. Balogh, R. Morris and W. Samotij, Independent sets in hypergraphs,
{\it  J. American Math. Soc.} {\bf 28} No. 3 (2015), 669--709.

\vspace*{-1ex}

\bibitem {Be58} C. Berge, Sur le couplage maximum d'un graphe, {\it
    Comptes rendus hebdomadaires des s\'{e}ances de l'Acad\'{e}mie des
    sciences} {\bf 247} (1958), 258--259.

\vspace*{-1ex}

\bibitem{BDE76} B. Bollob\'as, D.E. Daykin, and P. Erd\H{o}s, Sets of independent edges of a hypergraphs, \emph{Quart. J. Math. Oxford Ser.}, \textbf{27} (1976), 25--32.
\vspace*{-1ex}
\bibitem{Er65}P. Erd\H{o}s,   A problem on independent $r$-tuples,
\emph{Ann. Univ. Sci. Budapest. E\"otv\"os Sect. Math.}, \textbf{8}
(1965), 93--95.
\vspace*{-1ex}

\bibitem{EKL18} D. Ellis, N. Keller, and N. Lifshitz, Stability versions of Erd\H{o}s-Ko-Rado type theorems, via isoperimetry,
arXiv:1604.02160v4 [math.CO].

\vspace*{-1ex}
\bibitem{Fr13} P. Frankl, Improved bounds for Erd\H{o}s matching conjecture, {\it J. Combin. Theory Ser. A}, {\bf 120} (2013), 1068--1072.
\vspace*{-1ex}
\bibitem {FK18} P. Frankl and A. Kupavskii, The  Erd\H{o}s matching conjecture and concentration inequalities, arXiv:1806.08855v2 [math.CO].
\vspace*{-1ex}
\bibitem{FR85} P. Frankl and V. R\"odl,  Near perfect coverings in graphs and hypergraphs, \emph{European J. Combin.,}  \textbf{6} (1985), 317--326.
\vspace*{-1ex}

\bibitem{Han15}
J. Han, Near perfect matching in $k$-uniform hypergraph, \emph{Combinatorics, Probability and Computing,} 24 (2015), 723--732.
\vspace*{-1ex}

\bibitem{HPS09} H. H\`an, Y. Person and M. Schacht, On Perfect Matchings in
Uniform Hypergraphs with Large Minimum Vertex Degree, \emph{SIAM J.
Discret. Math.}, \textbf{23} (2009), 732--748.

\vspace*{-1ex}




\bibitem{Ka64}
G. Katona, Intersection theorems for systems of finite sets, Acta Math. Acad. Sci. Hungar. \textbf{15} (1964), 329--337.
\vspace*{-1ex}
\bibitem{Kh13}
I. Khan, Perfect matchings in 3-uniform hypergraphs with large vertex degree, \emph{SIAM J. Discrete Math.}, \textbf{27}
(2013), 1021--1039.
\vspace*{-1ex}


\bibitem{KO} D. K\"uhn and D. Osthus, Matchings in hypergraphs of large
minimum degree, \emph{J. Graph Theory}, \textbf{51} (2006),
269--280.
\vspace*{-1ex}

\bibitem{KOT14} D. K\"uhn, D. Osthus and T. Townsend, Fractional and integer matchings in uniform  hypergraphs, \emph{European J. Combin.},
\textbf{38} (2014), 83--96.
\vspace*{-1ex}

\bibitem{KOT13} D. K\"uhn, D. Osthus and A. Treglown, Matchings in
3-uniform hypergraphs, \emph{J. Combin. Theory, Ser. B},
\textbf{103} (2013), 291--305.
\vspace*{-1ex}

\bibitem{Mi05} M. Mitzenmacher and E. Upfal, Probability and Computing: Randomized Algorithms and Probabilistic Analysis, \textit{Cambridge University Press} (2005) ISBN 0-521-83540-2.
\vspace*{-1ex}

\bibitem{Rod09} V. R\"odl, and A. Ruc\'iski. Dirac-type questions for hypergraphs-a survey (or more problems for Endre to solve), \emph{An Irregular Mind}, Springer Berlin Heidelberg (2010), 561--590.

\vspace*{-1ex}
\bibitem{RRS06} V. R\"odl, A. Ruci\'nski, E. Szemer\'edi, Perfect matchings in uniform hypergraphs with large minimum degree,
\emph{European J. Combin.}, \textbf{27} (2006), 1333--1349.
\vspace*{-1ex}

\bibitem{RRS09} V. R\"odl, A. Ruci\'nski, E. Szemer\'edi, Perfect matchings in large uniform hypergraphs with large minimum
collective degree, \emph{J. Combin. Theory Ser. A}, \textbf{116} (2009), 616--636.
\vspace*{-1ex}

\bibitem{ST15} D. Saxton and A. Thomason, Hypergraph containers, {\it
    Invent. Math.} {\bf 201} (3) (2015) 925--992.
\vspace*{-1ex}
\bibitem{TZ12} A. Treglown and Y. Zhao, Exact minimum degree thresholds for perfect matchings in uniform hypergraphs I,
\emph{J. Comb. Theory Ser. A}, \emph{119} (2012),  1500--1522.
\vspace*{-1ex}
\bibitem{TZ13} A. Treglown and Y. Zhao, Exact minimum degree thresholds for perfect matchings in uniform hypergraphs II,
\emph{J. Comb. Theory Ser. A}, \emph{120} (2013),  1463--1482.

\end{thebibliography}
\end{document}